\documentclass[aop]{imsart}

\RequirePackage{amsthm,amsmath,amsfonts,amssymb}
\RequirePackage[numbers]{natbib}
\RequirePackage[colorlinks,citecolor=blue,urlcolor=blue]{hyperref}
\RequirePackage{graphicx}
\RequirePackage{enumitem}

\startlocaldefs

\usepackage[dvipsnames]{xcolor}
\usepackage{caption}
\usepackage{subcaption}
\numberwithin{equation}{section}

\definecolor{col1}{rgb}{0.6, 0.7, 0.8}
\definecolor{col2}{rgb}{0.7, 0.8, 0.65}
\definecolor{col3}{rgb}{0.8, 0.9, 0.5}
\definecolor{col4}{rgb}{0.91,0.94, 0.53}
\definecolor{col5}{rgb}{0.98,0.99,0.6}

\definecolor{c4}{rgb}{0.58, 0.69, 0.62}
\definecolor{c3}{rgb}{0.64, 0.76, 0.68}
\definecolor{c2}{rgb}{0.74, 0.86, 0.78}
\definecolor{c1}{rgb}{0.84, 0.96, 0.88}

\definecolor{bondiblue}{rgb}{0.0, 0.58, 0.71}

\definecolor{d1}{rgb}{0.6, 0.6, 0.6}

\definecolor{d2}{rgb}{0.35, 0.68, 0.91}
\definecolor{d3}{rgb}{0.43, 0.62, 0.99}
\definecolor{d3a}{rgb}{0, 0.75, 0.99}
\definecolor{d4}{rgb}{0.2, 0.45, 0.8}
\definecolor{d5}{rgb}{0.49, 0.67, 0.84}
\definecolor{d6}{rgb}{0.3, 0.4, 0.9}
\definecolor{d7}{rgb}{0.35, 0.9, 0.95}
\definecolor{d8}{rgb}{0.6, 0.8, 0.95}

\definecolor{d2}{rgb}{0.3, 0.68, 0.8}

\definecolor{textcol}{rgb}{0.37,0,0.57}
\usepackage{tikz}
\usepackage{forest}
\forestset{
  default preamble={
   for tree={circle, draw, 
            minimum size=2.1em, 
            inner sep=1pt} 
  }
}

\usetikzlibrary{positioning}

\theoremstyle{plain}
\newtheorem{thm}{Theorem}[section]
\newtheorem{lemma}[thm]{Lemma}
\newtheorem{cor}[thm]{Corollary}
\newtheorem{conj}[thm]{Conjecture}
\newtheorem{proposition}[thm]{Proposition}
\theoremstyle{remark}
\newtheorem{df}[thm]{Definition}

\newcommand{\ordS}{\vec{\mathbb{S}}}
\endlocaldefs

\begin{document}

\begin{frontmatter}
\title{Random planar trees and the Jacobian conjecture}
\runtitle{Random planar trees and the Jacobian conjecture}

\begin{aug}
\author[A]{\fnms{Elia}~\snm{Bisi}\ead[label=e1]{elia.bisi@tuwien.ac.at}},
\author[B]{\fnms{Piotr}~\snm{Dyszewski}\ead[label=e2]{pdysz@math.uni.wroc.pl}},
\author[C]{\fnms{Nina}~\snm{Gantert}\ead[label=e3]{nina.gantert@tum.de}},\\
\author[D]{\fnms{Samuel G. G.}~\snm{Johnston}\ead[label=e4]{samuel.g.johnston@kcl.ac.uk}},
\author[E]{\fnms{Joscha}~\snm{Prochno}\ead[label=e5]{joscha.prochno@uni-passau.de}},
\and
\author[F]{\fnms{Dominik}~\snm{Schmid}\ead[label=e6]{d.schmid@uni-bonn.de}}

\address[A]{Institute of Statistics and Mathematical Methods in Economics,
TU Vienna\printead[presep={,\ }]{e1}}

\address[B]{Mathematical Institute,
Wroc\l{}aw University \printead[presep={,\ }]{e2}}

\address[C]{Department of Mathematics,
TU Munich \printead[presep={,\ }]{e3}}

\address[D]{Department of Mathematics,
King's College London \printead[presep={,\ }]{e4}}

\address[E]{Faculty of Computer Science and Mathematics,
University of Passau \printead[presep={,\ }]{e5}}

\address[F]{Mathematical Institute, University of Bonn \printead[presep={,\ }]{e6}}

\end{aug}

\begin{abstract}
We develop a probabilistic approach to the {celebrated Jacobian conjecture, which states that} any {Keller map} (i.e.\ any polynomial mapping $F\colon \mathbb{C}^n \to \mathbb{C}^n$ whose Jacobian determinant is a nonzero constant) has a compositional inverse which is also a polynomial. 
The Jacobian conjecture may be formulated in terms of a problem involving labellings of rooted trees; we give a new probabilistic derivation of this formulation using multi-type branching processes.

Thereafter, we develop a simple and novel approach to the Jacobian conjecture in terms of a problem involving shuffling subtrees of $d$-Catalan trees, i.e.\ planar $d$-ary trees.
We also show that, if one can construct a certain Markov chain on large $d$-Catalan trees which updates its value by randomly shuffling certain nearby subtrees, and in such a way that the stationary distribution of this chain is uniform, then the Jacobian conjecture is true.

Finally, we use the local limit theory of large random trees to show that the subtree shuffling conjecture is true in a certain asymptotic sense, and thereafter use our machinery to prove an approximate version of the Jacobian conjecture, stating that inverses of Keller maps have small power series coefficients for their high degree terms.

\end{abstract}



\end{frontmatter}

\section{Introduction and overview}
\label{sec:intro}

\subsection{The Jacobian conjecture}
\label{subsec:jacobianIntro}

The Jacobian conjecture, a problem from algebraic geometry which is concerned with the invertibility of polynomial mappings in $n$ variables over a field of characteristic zero,
is one of the outstanding open problems in all of mathematics.
It was first conjectured by Ott-Heinrich Keller in 1939 \cite{keller1939} and has remained unsolved for over 80 years; even the case $n=2$ is still open.
It features as Problem 16 on Smale's list of \emph{18 mathematical problems for the $21^{\text{st}}$ century} \cite{smale} and a resolution seems to require completely novel ideas and methods.

We now give a brief account of the Jacobian conjecture, assuming no specialist knowledge of commutative algebra or algebraic geometry. For the sake of simplicity we work over $\mathbb{C}$, though this formulation is equivalent to more general formulations over fields of characteristic zero \cite{VDEbook}.
Consider the ring $\mathbb{C}[X_1,\ldots,X_n]$ of polynomials in $n$ variables with coefficients in $\mathbb{C}$.
A \textbf{polynomial mapping} $F = (F_1,\ldots,F_n)$ is a function $F\colon\mathbb{C}^n \to \mathbb{C}^n$ such that each component $F_i$ is an element of $\mathbb{C}[X_1,\ldots,X_n]$.
A polynomial mapping $F$ is said to be a \textbf{polynomial automorphism} if there exists a second polynomial mapping $G$ such that $F \circ G = G \circ F = {I}$, where ${I}$ is the identity mapping on $\mathbb{C}^n$. 
We define the \textbf{Jacobian matrix} of $F$, denoted $JF$, to be the $n \times n$ matrix whose $(i,j)^{\text{th}}$ entry is $(JF)_{i,j} := \frac{ \partial}{\partial X_j} F_i$.
We denote by $\det JF$ the \textbf{Jacobian determinant} of $F$, which is itself an element of $\mathbb{C}[X_1,\ldots,X_n]$.

We say a polynomial mapping is a \textbf{Keller mapping} if its Jacobian determinant is equal to a non-zero constant in $\mathbb{C}$.
It is easily verified that every polynomial automorphism is a Keller mapping.
The celebrated Jacobian conjecture is the converse of this statement.

\begin{conj}[The Jacobian conjecture] \label{conj:JC}
Every Keller mapping is a polynomial automorphism.
\end{conj}

There is a vast literature on the conjecture and its history.
We refer the reader to van den Essen's monograph \cite{VDEbook} for more information.

\subsection{The combinatorial approach to the Jacobian conjecture} \label{sec:combapproach}
While the Jacobian problem most commonly lies in the remit of algebraic geometry, over the years several authors have noted that the problem may be attacked using combinatorial arguments {\cite{AA, JP, singer1,wright05a, zeilberger, zhao, zhao2}}. The combinatorial formulation of the Jacobian conjecture was crystallized in recent work by the fourth and fifth authors \cite{JP}, and we now give a brief outline of this very concrete approach, which rests on the following three observations:

\begin{itemize}
\item We begin by appealing to the celebrated Bass--Connell--Wright reduction of the Jacobian conjecture \cite{BCW}: to prove the Jacobian conjecture it is sufficient to fix $d \geq 3$ and establish, for all $n$, the polynomial invertibility of every map of the form $F = {I} - H:\mathbb{C}^n \to \mathbb{C}^n$, where $H$ is $d$-homogeneous and the Jacobian matrix $JH$ is \textbf{nilpotent}, i.e.\ $(JH)^p = 0$ for some $p \leq n$.

For any such map we write
\begin{align} \label{eq:BCWform}
F_i(X_1,\ldots,X_n) := X_i - \sum_{|\alpha| = d } \frac{ H_{i,\alpha}}{\alpha!} X^\alpha
\end{align}
for the $i^{\text{th}}$ component of $F$, where the sum is taken over multi-indices $\alpha=(\alpha_1, \ldots, \alpha_n) \in \mathbb{Z}_{ \geq 0}^n$ with $|\alpha| := \alpha_1 + \dots + \alpha_n =d$, and we set $\alpha! := \alpha_1!\cdots \alpha_n!$ and $X^\alpha := X_1^{\alpha_1} \cdots X_n^{\alpha_n}$.
As such, we canonically associate $F$ (or $H$) with the set of complex coefficients $H = (H_{i,\alpha}\colon i \in [n], |\alpha| =d)$.
\item Whenever $F = {I} - H$ is a polynomial mapping, where $H$ is homogeneous degree $d$ (and $JH$ may or may not be nilpotent), $F$ has a power series inverse $F^{-1}$ whose $i^{\text{th}}$ component takes the form
\begin{align*}
F^{-1}_i(X_1,\ldots,X_n) = X_i + \sum_{k \geq 1} \sum_{ |\alpha| = (d-1)k+1  } g_{i,\alpha} X^\alpha.
\end{align*}
Moreover, for $|\alpha| = (d-1)k+1$, the complex coefficients $g_{i,\alpha}$ may be written in terms of the formula \cite[Theorem 2.3]{JP}
\begin{align} \label{eq:Gsum}
g_{i,\alpha} = \Phi_i^\alpha(H) := \frac{1}{(d!)^k} \sum_{T \in \mathcal{C}_k^{(d)} } \sum_{ \text{$(i,\alpha)$ labellings $\mathcal{T}$ of $T$}} \mathcal{E}_H(\mathcal{T}).
\end{align}
\begin{figure}
  \centering   
\scalebox{.8}{
 \begin{forest}
[1, fill=c1
[1,  fill=c1
[2, fill=c2]
[4,fill=c4]
[1,fill=c1]
]
[2, fill=c2
[2,fill=c2]
[3,  fill=c3
[3, fill=c3]
[1,  fill=c1]
[2, fill=c2]
]
[3,  fill=c3]
]
[4, fill=c4
]
]
\end{forest}
}  
\captionsetup{width=.85\linewidth}
\caption{{An $(i,\alpha)$ labelling $\mathcal{T}$ of a tree $T\in\mathcal{C}_k^{(d)}$.
In the figure, the degree is $d=3$, $k=4$, the dimension is $n=4$, the root type is $i=1$ and the leaf type is $\alpha=(2,3,2,2)$.
The $H$-weight of the depicted labelled tree $\mathcal{T}$ is $\mathcal{E}_H(\mathcal{T}) = H_{1,(1,1,0,1)}^2 H_{2,(0,1,2,0)} H_{3,(1,1,1,0)}$.}
The coefficient $g_{i,\alpha}$ for the inverse of $F = {I} - H$ with $H$ homogeneous degree $d$ may be computed as a sum over $(i,\alpha)$ labellings of $d$-Catalan trees; see \eqref{eq:Gsum}.}\label{fig:mbt1}
\end{figure}
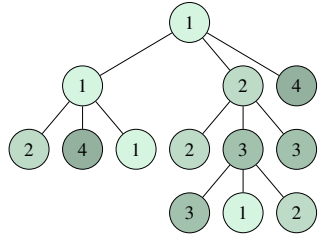
Here, $\mathcal{C}_k^{(d)}$ is the set of $d$-Catalan trees (i.e.\ planar, rooted, $d$-ary trees) with $(d-1)k+1$ leaves (i.e.\ with $k$ non-leaf vertices).
If $T = (V,E) \in\mathcal{C}_k^{(d)}$, then a labelling $\mathcal{T}$ of $T$ is a pair $\mathcal{T} = (T,\tau)$ where $\tau:V \to [n]$ is a function giving vertex $v$ \textbf{type} $\tau(v)$.
{For $i \in \mathbb{N}$ and $\alpha \in \mathbb{Z}^n_{\geq 0}$}  an $(i,\alpha)$ labelling of $T$ is any labelling of $T$ giving the root type $i$ and giving exactly $\alpha_\ell$ of the leaves type $\ell$, for all $\ell\in[n]$.
Finally, if for any non-leaf vertex $v$ we let $\mu(v)$ be the multi-index whose $j^{\text{th}}$ component counts the number of type $j$ children of $v$ (so that $|\mu(v)|=d$), then the \textbf{$H$-weight} $\mathcal{E}_H(\mathcal{T})$ of the labelling $\mathcal{T}$ of $T$ is given by 
\begin{align} \label{eq:blender}
\mathcal{E}_H(\mathcal{T}) := \prod_{v \in V \text{ non-leaf}} H_{\tau(v),\mu(v)},
\end{align}
where the product is taken over all non-leaf vertices in $V$.
See Figure~\ref{fig:mbt1} on page \pageref{fig:mbt1}.

\item The nilpotency of $JH$ imposes a condition on the coefficients of $H$, which we now describe.
We define a $d$-ary \textbf{fern} to be an element of $\mathcal{C}_n^{(d)}$ of height $n$, with a designated `sink' vertex in generation $n$, and the property that at most one vertex in each generation is not a leaf; see Figure \ref{fig:mbt2}.
{Fix an arbitrary $d$-ary fern $U$.}
It then turns out that the nilpotency of $JH$ is equivalent to the statement
\begin{align} \label{eq:nilp0}
\Psi_{i,j}^\alpha(H) = 0 \qquad \text{for every $i,j \in [n]$ and $\alpha \in \mathbb{Z}_{\geq 0}^n$ with $|\alpha| = (d-1)n$},
\end{align}
where
\begin{align} \label{eq:nilp}
\Psi_{i,j}^\alpha(H) := \sum_{ \text{$(i,j,\alpha)$ labellings $\mathcal{U}$ of $U$}} \mathcal{E}_H(\mathcal{U}).
\end{align}
Here, an $(i,j,\alpha)$ labelling of $U$ is any labelling of $U$ giving the root type $i$, the sink type $j$, and giving exactly $\alpha_\ell$ of the non-sink leaves type $\ell$, for all $\ell\in[n]$; see \cite[Lemma 2.8]{JP}.
{It can be easily seen that definition~\eqref{eq:nilp} does not depend on the choice of the $d$-ary fern $U$.}
\end{itemize}
\begin{figure}
  \centering   
\scalebox{.8}{
 \begin{forest}
[1, fill=c1
[4, fill=c4]
[1, fill=c1]
[1,  fill=c1
[2, fill=c2
[1, fill=c1
]
[2, fill=c2
[3, fill=c3]
[3, fill=c3]
[1, fill=c1]
]
[4, fill=c4]
]
[4,fill=c4]
[1,fill=c1]
]
]
\end{forest}
}  
\captionsetup{width=.85\linewidth}
\caption{An $(i,j,\alpha)$ labelling of a $3$-ary fern of height $n=4$, with root type $i=1$, sink type $j=3$ {and non-sink leave type $\alpha=(4,0,1,3)$}.
If $JH$ is nilpotent, then weighted sums over $(i,j,\alpha)$ labellings of a fixed $d$-ary fern of length $n$ are zero; see \eqref{eq:nilp0}-\eqref{eq:nilp}.
} \label{fig:mbt2}
\end{figure}
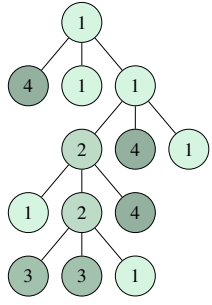
We prove \eqref{eq:Gsum} and the equivalence of \eqref{eq:nilp0} with the nilpotency of $JH$ in the sequel.

To recapitulate, if $H$ is degree $d$ homogeneous, then the coefficients of the inverse mapping for $F = {I} -H$ are given in \eqref{eq:Gsum} as a sum over weighted trees; see Figure \ref{fig:mbt1}.
If we assume that $JH$ is nilpotent, then the coefficients of $H$ must satisfy the nilpotency condition \eqref{eq:nilp0}-\eqref{eq:nilp} involving sums over labellings of a $d$-ary fern; see Figure \ref{fig:mbt2}.

Now, as mentioned above, the Jacobian conjecture is equivalent to the statement that every $F$ of the form $F= {I}-H$ with $H$ $d$-homogeneous and $JH$ nilpotent has the property that $F^{-1}$ is a polynomial mapping.
In other words, for such $F$, it must be the case that the combinatorial quantity $\Phi_i^\alpha(H)$ defined in \eqref{eq:Gsum} for the coefficients of $F^{-1}$ is zero for all but finitely many choices of $\alpha$.
Combining all the above observations, we arrive at the following combinatorial formulation of the Jacobian conjecture:

\begin{conj}[Jacobian conjecture -- combinatorial version] \label{conj:JC2}
There exists $d \geq 3$ such that the following is true for all $n \geq 1$.
Suppose $H = (H_{i,\alpha}\colon i \in [n], |\alpha| = d)$ is a collection of complex numbers satisfying $\Psi_{i,j}^\alpha(H) = 0$ for all $i,j \in [n]$ and $\alpha \in \mathbb{Z}_{\geq 0}^n$ with $|\alpha| = (d-1)n$.
Then, for each $i\in [n]$, $\Phi_i^\alpha(H) = 0$ for all but finitely many $\alpha$.
\end{conj}

We stress that Conjecture \ref{conj:JC2} is equivalent to the Jacobian conjecture, Conjecture \ref{conj:JC}.

In this article we give a rapid new derivation of the combinatorial formulation of the Jacobian conjecture \cite{JP} using a probabilistic approach centered around multi-type Galton--Watson trees. 
We then substantially reduce the complicated conjecture in \cite{JP} involving labelled trees to a new and far simpler conjecture involving unlabelled trees. We show it is also possible to reformulate the latter in terms of stationary distributions of Markov chains on trees that update by shuffling certain subtrees. Finally, we apply the local limit theory of large random trees to attack this simpler conjecture, showing it is true in a certain asymptotic sense, and relating this approximate form back to the original Jacobian conjecture. In short, probabilistic machinery has the potential to tackle the Jacobian conjecture, and thus the \emph{raison d'etre} of the present article is to attract probabilists to the problem.
 
In Section \ref{subsec:results} we give an outline of our main results, with the understanding that full details surrounding the definitions will be given in the sequel.

\subsection{Main results} \label{subsec:results}

The basic object we study are \textbf{$d$-Catalan trees}: these are rooted, planar, $d$-ary trees.
Here and throughout, by \textbf{$d$-ary tree} we mean a tree in which each vertex has either $0$ or $d$ children.
We write 
\begin{align*}
\mathcal{C}_k^{(d)} := \{ \text{$d$-Catalan trees with $k$ internal vertices} \}.
\end{align*}
Take an ancestral path $(v_0,\ldots,v_p)$ of length $p$ in a tree $T \in \mathcal{C}_k^{(d)}$, and consider the subtrees subtended by the $(d-1)p$ siblings of $v_1,\ldots,v_p$ (some of which may be leaves). 
The \textbf{shuffle class} associated with this ancestral path is the subset of $\mathcal{C}_k^{(d)}$ consisting of trees $T'\in \mathcal{C}_k^{(d)}$ that can be obtained from $T$ by rearranging these $(d-1)p$ trees. We call any such subset a length-$p$ shuffle class. A length-$p$ shuffle class in $\mathcal{C}_k^{(d)}$ may contain up to $((d-1)p)!$ different trees $T'$ of $\mathcal{C}_k^{(d)}$, but will contain fewer than this number if some of the subtended subtrees are identical. See Section \ref{subsec:shuffles} for more precise definitions.


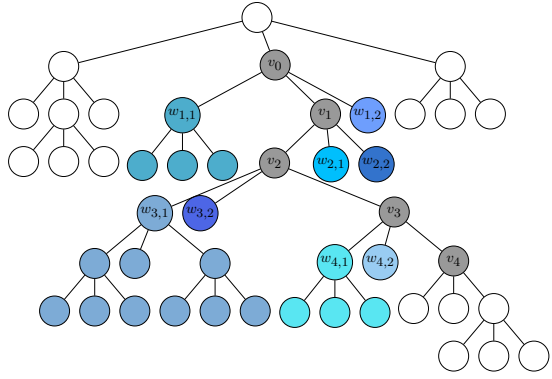
\begin{figure}
  \centering   
\scalebox{.6}{
 \begin{forest}
[
[
[]
[[][][]]
[]
]
[$v_0$, fill=d1
[$w_{1,1}$, fill=d2
[,fill=d2]
[,fill=d2]
[,fill=d2]
]
[$v_1$, fill=d1
[$v_2$, fill=d1
[$w_{3,1}$, fill=d5
[, fill=d5[, fill=d5][, fill=d5][, fill=d5]][, fill=d5][, fill=d5[, fill=d5][, fill=d5][, fill=d5]]]
[$w_{3,2}$, fill=d6]
[$v_3$,fill=d1
[$w_{4,1}$, fill=d7
[, fill=d7][, fill=d7][, fill=d7]]
[$w_{4,2}$, fill=d8]
[$v_4$,fill=d1
[][][[][][]]]
]
]
[$w_{2,1}$, fill=d3a]
[$w_{2,2}$, fill=d4]
]
[$w_{1,2}$, fill=d3]
]
[
[][][]]
]
\end{forest}
}  
\captionsetup{width=.85\linewidth}
\caption{A $d$-Catalan tree with $d=3$ and $k=15$ internal vertices.
An ancestral path $(v_0,\dots,v_4)$ of length $p=4$ is highlighted in grey.
Each vertex $v_i$, $i=1,\dots,p$, has siblings $w_{i,j}$, $j=1,\dots,d-1$.
Thus, there are $(d-1)p=8$ subtrees coming off the path, each of which is subtended by one of the $w_{i,j}$; these subtrees are highlighted in various shades of blue.
Five of these trees are leaves; two of them (those subtended by $w_{1,1}$ and $w_{4,1}$) are the unique $3$-Catalan tree with one internal vertex; finally, the one subtended by $w_{3,1}$ is a $3$-Catalan tree with three internal vertices.
The associated shuffle class consists of the set of $8!/(5!2!1!)$ trees that may be obtained by rearranging the positions of the subtrees subtended by the $w_{i,j}$ whilst keeping the positions of the white and grey vertices fixed.
The notation of the $v_i$ and $w_{i,j}$ corresponds to that of Section \ref{subsec:shuffles}.}\label{fig:shuffle}
\end{figure}

We make the following conjecture about the vector space of functions on $\mathcal{C}_k^{(d)}$.

\begin{conj}[{Subtree} shuffling conjecture] \label{conj:shuffle}
Fix $d \geq 2$ and $p \geq 1$.
Then, there exists an integer $k_0(d,p)$ such that, for all $k \geq k_0(d,p)$, the constant function $\mathbf{1}$ on  $\mathcal{C}_n^{(d)}$ lies in
\[\mathrm{span} \{ \mathbf{1}_\mathcal{S} \colon \mathcal{S}\text{ is a length-$p$ shuffle class in $\mathcal{C}_k^{(d)}$} \},
\]
i.e.\ in the span of the indicator functions of the length-$p$ shuffle classes. 
\end{conj}

Let us interpret this conjecture.
Spelling out the statement, we are saying that to each length-$p$ shuffle class $\mathcal{S} \subseteq \mathcal{C}_k^{(d)}$ we may associate a real number $\lambda_\mathcal{S}$ in such a way that for every tree $T \in \mathcal{C}_k^{(d)}$ we have
\begin{align} \label{eq:preeq}
\sum_{ \mathcal{S} \text{ length-$p$ shuffle class}} \lambda_\mathcal{S} \mathbf{1}_S(T) = 1.
\end{align}
We can think of $\lambda_\mathcal{S}$ as a `score' given to each tree $T$ in $\mathcal{C}_k^{(d)}$ for membership in $\mathcal{S}$.
Conjecture~\ref{conj:shuffle} then states that, when the number of vertices is large, the shuffle classes form a sufficiently rich collection of subsets of $\mathcal{C}_k^{(d)}$ so that there exists an allocation of scores to shuffle classes for which each tree gets the same total score.

The reason for our interest in Conjecture \ref{conj:shuffle} is the following theorem, which we will prove in Section \ref{subsec:orthogonality}:

\begin{thm}[Subtree shuffling conjecture implies Jacobian conjecture] \label{thm:relation}
If, for some $d \geq 3$, Conjecture \ref{conj:shuffle} is true for all $p \in \mathbb{N}$, then so is the Jacobian conjecture.
\end{thm}

It is possible to adapt Conjecture \ref{conj:shuffle} and the subsequent Theorem \ref{thm:relation} to a problem involving tree-valued Markov chains, as we now outline.
A \textbf{$p$-shuffle chain} is a Markov chain $(T_j)_{j \geq 0}$ taking values in $\mathcal{C}_k^{(d)}$ that updates its value from $T_j$ to $T_{j+1}$ by choosing a length-$p$ ancestral path within $T_j$ according to some random $T_j$-dependent rule, and then creating $T_{j+1}$ from $T_j$ by choosing uniformly a rearrangement of the subtrees coming off this ancestral path.
See Section \ref{subsec:markov} for the precise definition.

\begin{conj}[Shuffle chain conjecture] \label{conj:shufflechain}
Fix $d \geq 2$ and $p \geq 1$.
Then, there exists an integer $k_0(d,p)$ such that, for all $k \geq k_0(d,p)$, there exists a $p$-shuffle chain whose stationary distribution is the uniform distribution on $\mathcal{C}_k^{(d)}$.
\end{conj}
We will show the following.
\begin{thm}[Shuffle chain conjecture implies Jacobian conjecture] \label{thm:shufflechain}
If, for some $d\geq 3$, Conjecture \ref{conj:shufflechain} is true for all $p \geq 1$, then so is the Jacobian conjecture.
\end{thm}

We now turn to outlining some of our unconditional results.
In this direction, we say a $d$-Catalan tree $T$ is \textbf{$p$-perfect} if it contains a \textbf{$p$-perfect path}, i.e. an ancestral path such that each sibling off the path is a leaf.
See Figure \ref{fig:preperfect}.

\begin{figure}
  \centering   
\scalebox{.6}{
 \begin{forest}
[
[][[[[][][]][][[][][]]][][]][,fill=d1[,fill=d1[,fill=d5][,fill=d5][,fill=d1[,fill=d5][,fill=d5][,fill=d1[,fill=d5][,fill=d1[,fill=d1[[][][[][][]]][][[[][][]][][[][][]]]][,fill=d5][,fill=d5]][,fill=d5]]]][,fill=d5][,fill=d5]]
]
\end{forest}
}  
\captionsetup{width=.85\linewidth}
\caption{A $p$-perfect path of length $p=5$ in a $3$-Catalan tree.
Every sibling off the path is a leaf.}\label{fig:preperfect}
\end{figure}
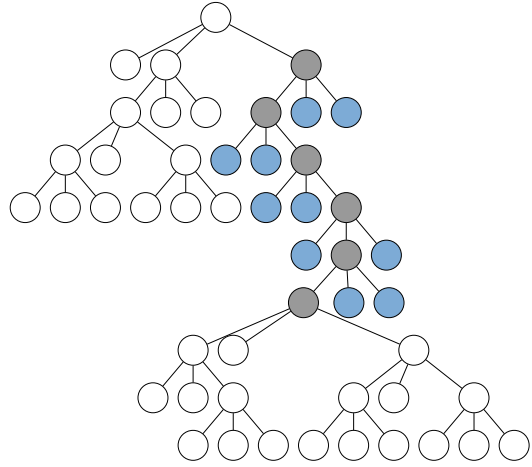

In Section \ref{subsec:likelihoodPerfectTrees} we appeal to methods from the local limit theory of large random trees to show that, when the number of vertices is large, $p$-perfect trees are overwhelmingly common.
Indeed, there is an underlying philosophy in the local behaviour of large random trees that distinct parts of the tree behave more or less independently of one another.
With this picture in mind, the probability that a tree avoids having a $p$-perfect path somewhere decays exponentially in the number of vertices.
This motivates the following result, which we will prove using a simple algorithm for generating uniform $d$-Catalan trees.
\begin{thm} \label{thm:perfect}
{Let $p\geq 1$} and let $\mathbb{P}_k^{(d)}$ be the uniform measure on $\mathcal{C}_k^{(d)}$. 
Then there exists a constant $\kappa_{p,d}>0$ such that
\begin{align*}
\mathbb{P}_k^{(d)}( \text{$T$ is not $p$-perfect} ) \leq e^{ - \kappa_{p,d} {(k-p)_+} } \qquad \text{for all } k\in\mathbb{N},
\end{align*}
where $x_+:=\max(x,0)$.
In particular, we may set $\kappa_{p,d} := (2 p d^pe^{p})^{-1}$.
\end{thm}

If a tree is $p$-perfect, then rearranging the subtrees coming off its perfect path keeps the tree unchanged.
Thus, $p$-perfect trees lie in a singleton shuffle class, and, in the setting of Conjecture \ref{conj:shuffle}, we can adjust the parameteres $\lambda_\mathcal{S}$ so that \eqref{eq:preeq} is true at least for all perfect trees, i.e.\ for a fraction of at least $(1 -  e^{ - \kappa_{p,d} {(k-p)_+} } )$ of trees in $\mathcal{C}_k^{(d)}$.
In fact, we are also able to prove that, up to an error of the order $e^{ - O(\log^2 k)}$, the equation \eqref{eq:preeq} is true for \emph{all} trees in $\mathcal{C}_k^{(d)}$.
Let us summarise these facts in the following approximate version of Conjecture~\ref{conj:shuffle}, which states that the constant function can be well approximated, in multiple ways, by a function in the span of shuffle indicators.

\begin{thm}[Approximating the constant function] \label{thm:approx}
Let $p\in\mathbb N$ {and $d\geq 2$}.
For $k\in\mathbb N$ sufficiently large, there exists a function $\phi_k$ in $\mathrm{span} \{ \mathbf{1}_\mathcal{S} \colon \mathcal{S}\text{ is a length-$p$ shuffle class in $\mathcal{C}_k^{(d)}$} \}$ that well approximates the constant function $\mathbf{1}$, in the sense that the following bounds hold:
\begin{itemize}
\item \emph{($\ell^1$ bound)} the exponential total variation bound
\begin{align} \label{eq:xx}
|| \phi_k - \mathbf{1}||_1 := \frac{1}{\# \mathcal{C}_k^{(d)} } \sum_{T \in \mathcal{C}_k^{(d)} } | \phi_k(T) - 1 | &\leq C e^{ - c k};
\end{align}
\item \emph{($\ell^\infty$ bound)}
the superpolynomial uniform bound
\begin{align} \label{eq:yy} 
|| \phi_k - \mathbf{1}||_\infty :=  \sup_{T \in \mathcal{C}_k^{(d)} } | \phi_k(T) - 1 |   &\leq C' e^{ - c' (\log k)^2}.
\end{align}
\end{itemize}
Here $C,c,C',c'\in(0,\infty)$ are constants that depend on the degree $d$ of the trees and on the length $p$ of the shuffle classes, but not on the number $k$ of internal vertices in the trees.
\end{thm}

Notice that the $1$-norm of \eqref{eq:xx} can be also interpreted in terms of the uniform measure $\mathbb{P}_k^{(d)}$ on $\mathcal{C}_k^{(d)}$ defined in Theorem \ref{thm:perfect}.
Indeed, if we denote by $\mathbb{E}_k^{(d)}$ the expectation with respect to $\mathbb{P}_k^{(d)}$, then $|| f ||_1= \mathbb{E}_k^{(d)} |f(T)|$ for any function $f\colon \mathcal{C}_k^{(d)}\to\mathbb{R}$.

We finally tie this result back to the Jacobian conjecture: we leverage Theorem \ref{thm:approx} to provide quantitative information about the inverses of Keller maps, concluding that these have small power series coefficients for their high degree terms. 
The precise statement of this approximate version of the Jacobian conjecture can be found in Theorem \ref{thm:twobounds}.

\subsection{Relevant literature} \label{subsec:literature}
In this section we touch on some related work.

{To attack the Jacobian conjecture,} numerous authors have studied the polynomial invertibility of Keller mappings of various degrees in various dimensions.
Notably, Wang \cite{wang} proved in 1982 that the Jacobian conjecture is true in the special quadratic case:

\begin{thm}[Wang's Theorem \cite{wang}]\label{thm:wang}
Every quadratic Keller mapping is a polynomial automorphism.
\end{thm}
See \cite{BCW} for a simple proof.
Wang's theorem contrasts strikingly with the degree $d \geq 3$ reduction of Bass, Connell and Wright, which states that, in order to prove the Jacobian conjecture, it is sufficient to study degree $d$ Keller mappings {for a fixed $d\geq 3$}.

Notably, Wang's theorem implies that the analogue of the combinatorial formulation of the Jacobian conjecture (Conjecture \ref{conj:JC2}), \emph{but with degree $d =2$ rather than $d \geq 3$}, is true!
To this day, there is no known direct combinatorial derivation of this result.
The authors are hopeful that a combinatorial proof of Wang's theorem may lead to information on how to prove (or disprove) the more general Jacobian conjecture by combinatorial means.

Combinatorial approaches to the Jacobian conjecture in the vein of Section \ref{sec:combapproach}  date back at least as far as the work of Bass, Connell and Wright \cite{BCW}.
Wright has continued work on the conjecture from the perspective of inversion formulas in \cite{wright87,wright89,wright05a,wright05b}.
The explicit formulation involving weighted trees given in Conjecture \ref{conj:JC2} was given explicitly in \cite{JP}, but the idea is present in the aforementioned work of Wright, as well as in Singer \cite{singer1,singer2} and Zeilberger \cite{zeilberger}.
Let us also mention Abdesselam \cite{AA}, whose combinatorial approach draws on quantum field theory.

Singer \cite{singer1,singer2} first utilised `deweighting' techniques to convert the combinatorial formulation of the Jacobian problem involving weighted trees to one involving unweighted trees.
In \cite{singer1,singer2}, Singer was mainly concerned with studying the classical (i.e.\ $2$-ary) Catalan trees to obtain a combinatorial derivation of Wang's theorem involving shuffle classes of $2$-Catalan trees.
As we will explain in Section \ref{subsec:orthogonality}, our methods and techniques draw inspiration from Singer's approach.

Johnston \cite{johnston} recently used combinatorial methods to study the so-called strongly nilpotent automorphisms, i.e.\ those of the form $F = {I} - H$, where
\begin{align*}
JH(X^{(1)})\cdots JH(X^{(q)}) = 0 \qquad \forall X^{(1)},\ldots,X^{(q)} \in \mathbb{C}^n .
\end{align*}
The above property is called strong nilpotency of index $q$, as it is a strengthening of the nilpotency condition $(JH)^q = 0$.
Using combinatorial methods, \cite{johnston} showed that, if $F = {I} - H$ has degree $d$ and is strongly nilpotent of index $q$, then the inverse $F^{-1}$ of $F$ has degree at most $d^{q-1}$.

Finally, let us mention a conceptually related work of Le Jan and Sznitman \cite{LJS}, who used probabilistic expansions involving trees to study the three-dimensional Navier-Stokes equation. 

\subsection{Outline of the article}
\label{subsec:outline}
The remainder of this article is structured as follows.
\begin{itemize}
\item In Section \ref{sec:inversion}, we use arguments from branching processes to give a rapid proof of a combinatorial inversion formula for multivariate power series, which is in fact new in this generality.
\item In Section \ref{sec:Keller}, we show that the Keller property has a combinatorial interpretation in terms of shuffle classes of both labelled and unlabelled trees.
\item In Section \ref{sec:conjectures}, we prove Theorem \ref{thm:relation} and Theorem \ref{thm:shufflechain}.
\item In Section \ref{sec:exponential}, we study perfect trees, proving Theorem \ref{thm:perfect}.
This leads to the $\ell^1$ bound in Theorem \ref{thm:approx} and to an approximate version of the Jacobian conjecture (Theorem \ref{thm:twobounds}).
\item In Section \ref{sec:infinity}, we construct the $\ell^\infty$ approximation in Theorem \ref{thm:approx} and conclude the proof of the theorem.
\end{itemize}

Sections \ref{sec:inversion} and \ref{sec:Keller} contain new probabilistic proofs and insights, but are somehow preliminary to the rest of the work; they can be skipped on a first reading.
Sections \ref{sec:conjectures}, \ref{sec:exponential} and \ref{sec:infinity} are the core of this article.

\section{Combinatorial inversion via Galton--Watson trees} \label{sec:inversion}

In this section we study combinatorial inversion of power series.
The inversion formula we obtain is slightly more general than to the one appearing in \cite{JP}.
For the proof we use a novel and efficient probabilistic approach based on {multi-type} Galton--Watson trees.


We first settle some notation.
We will consider power series mappings $F=(F_1,\dots,F_n)$ with $i^{\text{th}}$ component of the form
\begin{align}
\label{eq:alphaFact}
F_i(X_1,\ldots,X_n) = X_i - \sum_{ \alpha \neq 0} h_{i,\alpha}X^\alpha = X_i - \sum_{ \alpha \neq 0} \frac{H_{i,\alpha}}{\alpha!}X^\alpha,
\end{align}
where the sum is taken over multi-indices $\alpha=(\alpha_1, \ldots, \alpha_n) \in \mathbb{Z}_{ \geq 0}^n\setminus \{0\}$, and we set $\alpha! := \alpha_1!\cdots \alpha_n!$ and $X^\alpha := X_1^{\alpha_1} \cdots X_n^{\alpha_n}$.
As such, we canonically associate $F$ with the set of complex coefficients $h=(h_{i,\alpha}\colon i \in [n],\alpha\neq 0)$ or $H = (H_{i,\alpha}\colon i \in [n],\alpha\neq 0)$.
We will always use an upper case $H_{i,\alpha}$ when the coefficients are divided through by $\alpha!$.
Working with the undivided coefficients is more amenable for probabilistic arguments, while working with the divided coefficients yields cleaner formulas when using the algebraic condition of nilpotency.
Therefore, we initially use the $h_{i,\alpha}$ to prove the inversion formula; we then rephrase the latter in terms of the $H_{i,\alpha}$ and continue with the upper case coefficients in the study of nilpotency in Section~\ref{sec:Keller}.

\subsection{Inversion formula for ordered trees}
\label{subsec:invFormula}
 
A labelled planar rooted {tree} $\mathcal{T}=(T,\tau)$ is a planar rooted tree $T=(V,E)$ together with a type function $\tau:V \to [n]$ giving the vertices types in $[n]$; we also call this an \textbf{$n$-type tree}.
We say such a tree is \textbf{ordered} if siblings of smaller types always lie to the left of siblings of larger types.

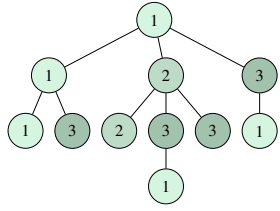
\begin{figure}
  \centering   
\scalebox{.7}{
 \begin{forest}
[1, fill=c1
[1,  fill=c1
[1, fill=c1]
[3,fill=c3]
]
[2, fill=c2
[2,fill=c2]
[3,  fill=c3
[1, fill=c1]
]
[3,  fill=c3]
]
[3, fill=c3
[1,  fill=c1]
]
]
\end{forest}
}  
\captionsetup{width=.85\linewidth}
\caption{An ordered $3$-type tree {in $\ordS_{1,(3,1,2)}$}.
}\label{fig:mbt}
\end{figure}

Given an $n$-type tree $\mathcal{T}$, we write $\mathrm{Leaftype}(\mathcal{T}) \in \mathbb{Z}_{\geq 0}^n$ for the multi-index with components
\begin{align*}
\mathrm{Leaftype}(\mathcal{T})_j := \# \{ \text{type $j$ leaves of $\mathcal{T}$} \}, \qquad j=1,\ldots,n.
\end{align*}
Given $i \in [n]$ and $\alpha \in \mathbb{Z}_{ \geq 0}^n$, let $\mathbb{S}_{i,\alpha}$ and $\vec{\mathbb{S}}_{i,\alpha}$ denote respectively the set of unordered and ordered $n$-type trees with root type $i$ and $\mathrm{Leaftype}(\mathcal{T}) := \alpha$ (i.e., the tree has $\alpha_j$ leaves of type $j$).
See Figure \ref{fig:mbt} for an example. In what follows, for $j$ in $[n]$, we will write $\mathbf{e}_j \in \mathbb{Z}_{ \geq 0}^n$ 
for the multi-index whose $j^{\text{th}}$ component is a $1$ and whose other components are $0$.
Let $(h_{i,\alpha} \colon i \in [n], \alpha \in \mathbb{Z}_{\geq 0}^n)$ be a collection of complex coefficients.
For an internal (i.e.\ non-leaf) vertex $v$, let $\mu(v) \in \mathbb{Z}_{\geq 0}^n$ be the multi-index whose $j^{\text{th}}$ component counts the number of type $j$ children of~$v$.
Define a second collection of coefficients by
\begin{align} \label{eq:Gco}
g_{i,\alpha} := \sum_{ \mathcal{T} \in \ordS_{i,\alpha}} \prod_{\substack{v \text{ internal} \\ \text{vertex of } \mathcal{T}}} h_{\tau(v),\mu(v)}.
\end{align}

We note that the sum in \eqref{eq:Gco} is over an infinite set, since, for any fixed $\alpha$, a tree in $\ordS_{i,\alpha}$ may contain an arbitrarily long path of vertices, each of whom has exactly one child.
We will assume for the remainder of this section that the linear coefficients $h_{i,\mathbf{e}_j}$ are such that
\begin{align} \label{eq:as1}
\sum_{k =0}^\infty A^k \text{ converges,} \quad \text{where } A=(A_{i,j})_{1\leq i,j\leq n}, \quad A_{i,j} := h_{i,\mathbf{e}_j} .
\end{align}
This assumption is sufficient to ensure that the sum in \eqref{eq:Gco} converges; see the discussion on alternating trees in \cite{JP} for details. 

We now state our main result of this section:

\begin{thm} \label{thm:inversion}
Suppose $(h_{i,\alpha} \colon i \in [n], \alpha \in \mathbb{Z}_{\geq 0}^n\setminus\{0\})$ is a collection of complex coefficients satisfying \eqref{eq:as1}, and that the power series mapping $F=(F_1,\dots,F_n)$ with $i^{\text{th}}$ component
\begin{equation} \label{eq:components}
F_i(X_1,\ldots,X_n) := X_i - \sum_{ \alpha \neq 0 } h_{i,\alpha}X^\alpha
\end{equation}
converges in an open neighbourhood of the origin in $\mathbb{C}^n$.
Then, $F$ has an analytic inverse $F^{-1} =(F^{-1}_1,\ldots,F^{-1}_n)$, with $i^{\text{th}}$ component given by
\begin{equation*}
F_i^{-1}(X_1,\ldots,X_n) := \sum_{ \alpha \neq 0} g_{i,\alpha} X^\alpha,
\end{equation*}
where $g_{i,\alpha}$ are defined in \eqref{eq:Gco}. 
\end{thm}

Theorem \ref{thm:inversion} is new in this exact formulation, but it is similar to countless other power series inversion formulas in the literature.
The special case of Theorem \ref{thm:inversion} in which each linear term $h_{i,\mathbf{e}_j}$ is zero is well trodden in the literature \cite{BCW,HS,wright87, wright89}.
An explicit formula for the case with non-zero linear {terms} only appeared recently in \cite{JP}, but the formulation therein involves \emph{alternating} trees and is more complicated.
We thus believe that the statement in Theorem \ref{thm:inversion} represents the most general and cleanest inversion formula for multivariate power series {so far}.

{The novelty here} lies in our probabilistic derivation, which is substantially more efficient than other approaches.
 Before proving Theorem \ref{thm:inversion} in Section \ref{subsec:inversionproof}, we first define and discuss multi-type Galton--Watson trees.

\subsection{Multi-type Galton--Watson trees} \label{subsec:GW}

In this section we define multi-type Galton--Watson trees and provide some background.

Recalling that $\ordS_{i,\alpha}$ denotes the set of ordered $n$-type trees with root type $i$ and leaf type $\alpha$, write
\begin{align*}
\ordS_{i} &:=\text{Finite ordered $n$-type trees with root type $i$} =  \bigcup_{\alpha \in \mathbb{Z}_{\geq 0}^n} \ordS_{i,\alpha}, \\
\ordS_{i,*} &:=  \text{Finite or infinite ordered $n$-type trees with root type $i$.}
\end{align*}

Take a collection $(h_{i,\alpha} \colon i \in [n], \alpha \in \mathbb{Z}_{\geq 0}^n)$ of complex numbers, indexed by $\alpha\in\mathbb{Z}_{\geq 0}^n$, as opposed to $\alpha\in\mathbb{Z}_{\geq 0}^n \setminus\{0\}$ as in the statement of Theorem \ref{thm:inversion}.
We say that this collection is a \textbf{multi-type offspring distribution} if $(h_{i,\alpha})_{\alpha \in \mathbb{Z}_{\geq 0}^n}$ is a probability mass function on $\mathbb{Z}_{\geq 0}^n$ for each $i$, i.e.\ if
\begin{align} \label{eq:prob}
h_{i,\alpha} \in [0,1] \quad \text{for each } (i,\alpha) \qquad \text{and} \qquad 
\sum_{\alpha \in \mathbb{Z}_{\geq 0}^n} h_{i,\alpha} = 1.
\end{align}
A multi-type Galton--Watson tree with root type $i$ is a random element of $\ordS_{i,*}$, defined by the following simple rule: the probability that a vertex of type $j$ has children of types $\beta$ is $h_{j,\beta}$.
More specifically, the tree may be constructed as follows.
We begin in generation $0$ with the root, which has type $i$.
Thereafter, we grow the labelled tree outwards from the root. Suppose we have grown up until generation $k$.
Then, for each vertex $v$ of some type $j$ in generation $k$, we choose a random element $\mu(v)$ in $\mathbb{Z}_{\geq 0}^n$ according to the probability mass function $(h_{j,\beta})_{\beta \in \mathbb{Z}_{\geq 0}^n}$; in particular, $h_{j,0}$ is the probability that $v$ is a leaf.
We then order the children of $v$ so that those of smaller types are placed to the left of those of larger types.
This creates a random ordered tree $\mathcal{T}$ in $\ordS_{i,*}$.
It is entirely possible (say, for instance, if $h_{j,0} = 0$ for each $j$) that with positive probability this process never terminates, and that the labelled tree we construct is infinite.

Let $\mathbf{P}_i^h$ denote the distribution of the Galton--Watson tree with root type $i$ and offspring distribution $(h_{i,\alpha})$.
Then, the probability of a tree $\mathcal{T}\in \ordS_{i,*}$ occurring is
\begin{align} \label{eq:treeprob}
\mathbf{P}_i^h(\mathcal{T}) = \prod_{v \text{ vertex of } \mathcal{T}} h_{\tau(v),\mu(v)},
\end{align}
where the product is over all {vertices of $\mathcal{T}$, including all  internal vertices and {leaves}.}

For $\mathcal{T}\in\ordS_{i,*}$, define the {associated process $(Z(k))_{k \geq 0}$ as} the $\mathbb{Z}_{\geq 0}^n$-valued sequence  given by 
\begin{align} \label{eq:process}
Z(k)_j := \# \{\text{vertices of $\mathcal{T}$ of type $j$ in generation $k$} \},
\end{align}
where generation $k$ of the tree refers to the set of vertices graph distance $k$ from the root.
Note that if $\mathcal{T}$ is finite, then $Z(k) = 0$ for all sufficiently large $k$.

Clearly, if a tree $\mathcal{T}$ in $\ordS_{i,*}$ is chosen randomly according to $\mathbf{P}_i^h$, then $(Z(k))_{k \geq 0}$ is a vector-valued Markov process under $\mathbf{P}_i^h$. Moreover, $(Z(k))_{k \geq 0}$ is a branching process. Indeed, let $Z^{j,\ell}(k+1)$ be the multi-index coding the children of the $\ell^{\text{th}}$ vertex of type $j$ in generation $k$, for all $k\geq 0$, $1\leq j\leq n$ and $1\leq \ell \leq Z(k)_j$. It is then clear that the sum of all $Z^{j,\ell}(k+1)$ over $j$ and $\ell$ codes the whole generation $k+1$:
\begin{align*}
Z(k+1) = \sum_{j=1}^n \sum_{\ell=1}^{Z(k)_j} Z^{j,\ell}(k+1) .
\end{align*}
Note that, conditionally on the event $\{ Z(k) = \alpha \}$,  there are $\alpha_j$ vertices of each type $j$ in generation $k$, and $(Z^{j,\ell}(k+1) \colon 1 \leq j \leq n, 1 \leq \ell \leq \alpha_j)$ are independent random variables such that each $Z^{j,\ell}(k+1)$ is distributed according to $(h_{j,\beta})_{\beta \in \mathbb{Z}_{\geq 0}^n}$. Equivalently, each $Z^{j,\ell}(k+1)$ has the same distribution as $\mu(v_{\mathrm{root}})$ in a Galton--Watson tree in which the root has type $j$.
The branching propery will be integral to our probabilistic proof of Theorem \ref{thm:inversion}.

While it will not be relevant to our goals in the remainder of the manuscript, let us close this section by discussing the asymptotic behaviour of multi-type Galton--Watson trees.
Let $\mathbf{E}^h_i$ be the expectation associated with the probability measure $\mathbf{P}^h_i$.
If we define
\begin{align*}
M_{i,j} := \mathbf{E}^h_i[Z(1)_j]
= \sum_{\alpha \in \mathbb{Z}_{\geq 0}^n} h_{i,\alpha} \alpha_j
\end{align*}
to be the expected number of type $j$ children had by a type $i$ vertex, then the asymptotic behaviour of $\mathbf{P}_i^h$ can be well described in terms of spectral properties of the mean matrix $M := (M_{i,j})_{1 \leq i,j \leq n}$.
Namely, if $M$ has an eigenvalue strictly larger than $1$ (\textbf{supercritical} case) and satisfies a certain irreducibility property, then a tree under $\mathbf{P}_i^h$ has a positive probability of being infinite and, accordingly, the sum of probabilities over finite trees satisfies
\begin{align*}
\sum_{ \mathcal{T}\in \ordS_{i} } \mathbf{P}_i^h(\mathcal{T}) < 1.
\end{align*}
On the other hand, whenever every eigenvalue of $M_{i,j}$ is less than one (\textbf{subcritical} case), the tree is guaranteed to be almost surely finite, i.e.\ 
\begin{align*}
\sum_{ \mathcal{T} \in \ordS_{i} } \mathbf{P}_i^h(\mathcal{T}) = 1.
\end{align*}
We refer the reader to \cite{AH} for further discussion of (multi-type) Galton--Watson trees.

\subsection{Proof of Theorem \ref{thm:inversion}} \label{subsec:inversionproof}

Let us give a brief overview of our proof of Theorem \ref{thm:inversion}.
We begin by appealing to the implicit function theorem to establish that $F$ has an analytic inverse $F^{-1}$ in a neighbourhood of the origin.
Thereafter we observe that the coefficients of degree $\leq k$ for the inverse can be expressed as polynomials in the coefficients of degree $\leq k$ in the original power series $F$.
We then appeal to analytic continuation, which allows us to assume that the coefficients of $F$ form a multi-type offspring distribution. 
Under this assumption, our proof of Theorem \ref{thm:inversion} involves studying a generating function of the leaf type of a multi-type Galton--Watson tree.
We utilise the branching property to show that this generating function satisfies a functional equation, which we then manipulate to prove Theorem \ref{thm:inversion}.

Recall that Theorem \ref{thm:inversion} considers a general collection $(h_{i,\alpha}\colon i \in [n],\alpha \in \mathbb{Z}_{\geq 0}^n\setminus\{0\} )$ of complex coefficients satisfying property \eqref{eq:as1}.
However, \textbf{until stated otherwise, we will assume for the remainder of this section that}
\begin{align} \label{eq:GWo}
h_{i,\alpha} \in [0,1] \quad \text{for each } (i,\alpha) \qquad \text{and} \qquad \sum_{\alpha \in \mathbb{Z}_{\geq 0}^n \setminus\{0\} } h_{i,\alpha} < 1.
\end{align}
We may naturally extend the coefficients to a collection also indexed by $\alpha = 0$ by setting 
\begin{align} \label{eq:assu}
d_i := h_{i,0} := 1 - \sum_{\alpha \in \mathbb{Z}_{\geq 0}^n \setminus\{0\} } h_{i,\alpha}.
\end{align}
It then follows from \eqref{eq:GWo} and \eqref{eq:assu} that $d_i > 0$ and that \eqref{eq:prob} holds, so that the collection $(h_{i,\alpha}\colon i \in [n],\alpha \in \mathbb{Z}_{\geq 0}^n )$ is a multi-type offspring distribution.

By \eqref{eq:treeprob} we have, for a given $\mathcal{T}\in \ordS_{i,*}$,
\begin{align*}
\mathbf{P}_i^h(\mathcal{T}) = \prod_{\substack{v \text{ internal} \\ \text{vertex of } \mathcal{T}}} h_{\tau(v),\mu(v)} \prod_{\substack{v \text{ leaf} \\ \text{vertex of } \mathcal{T}}} d_{\tau(v)} .
\end{align*}
Notice that
\begin{align*}
\prod_{\substack{v \text{ leaf} \\ \text{vertex of } \mathcal{T}}} d_{\tau(v)}  = d_1^{\alpha_1}\cdots d_n^{\alpha_n} =: d^\alpha
\qquad \text{for all } \mathcal{T} \in \ordS_{i,\alpha} .
\end{align*}
In particular, recalling the definition \eqref{eq:Gco} of $g_{i,\alpha}$, the probability that a multi-type Galton--Watson tree with root type $i$ is finite with leaf type $\alpha$ is given by 
\begin{align} \label{eq:probLeaves}
\mathbf{P}_i^h( \mathcal{T} \text{ is finite, } \mathrm{Leaftype}(\mathcal{T}) = \alpha ) = \sum_{\mathcal{T} \in \ordS_{i,\alpha}} \mathbf{P}_i^h(\mathcal{T}) = d^{\alpha} g_{i,\alpha}.
\end{align}

The key object we need is the generating function $E=(E_1,\dots,E_n)$ for the leaf type:
\begin{align*}
E_i(s_1,\ldots,s_n) :=\mathbf{E}_i^h\left[ \mathbf{1}_{\{\mathcal{T} \text{ finite}\}} s^{\mathrm{Leaftype}(\mathrm{\mathcal{T}})}\right]
= \sum_{\alpha\neq 0} \mathbf{P}_i^h( \mathcal{T} \text{ is finite, } \mathrm{Leaftype}(\mathcal{T}) = \alpha ) \, s^\alpha ,
\end{align*}
which is analytically well defined for $(s_1,\dots,s_n)$ in a neighbourhood of the origin in $\mathbb{C}^n$. 
It follows from \eqref{eq:probLeaves} that
\begin{align*}
E_i(s_1,\ldots,s_n)
= \sum_{\alpha\neq 0} g_{i,\alpha} d^\alpha s^\alpha
= g_i(d_1s_1,\dots,d_ns_n)
= g_i( D (s)) ,
\end{align*}
where $g=(g_1,\dots,g_n)$ and $D=(D_1,\ldots,D_n)$ are the maps defined by 
\[
g_i(X):= \sum_{\alpha\neq 0} g_{i,\alpha} X^{\alpha},
\qquad\quad
D_i(X_1,\ldots,X_n) := d_i X_i.
\]
In short, we have 
\begin{equation}
\label{eq:equalPowerSeries}
E = gD.
\end{equation}

We now use the branching property to derive a functional equation for $E$.
\begin{lemma} \label{lem:funktionell}
Suppose that $h=(h_1,\dots,h_n)$, with $h_i(X_1,\ldots,X_n) := \sum_{ \alpha \neq 0} h_{i,\alpha} X^\alpha$, satisfies \eqref{eq:GWo}.
Then, the leaf type generating function $E =(E_1,\ldots,E_n)$ satisfies the functional equation
\begin{align} \label{eq:mahler}
E = hE + D.
\end{align}
\end{lemma}
\begin{proof}
Recall the multi-type Galton--Watson process $(Z(k))_{k \geq 0}$ associated with a multi-type Galton--Watson tree and defined in \eqref{eq:process}. 

By conditioning on $Z(1)$, we have
\begin{align} \label{eq:O0}
E_i(s_1,\ldots,s_n)
&= \sum_{ \alpha \in {\mathbb{Z}_{\geq 0}^n}} \mathbf{P}_i^h(Z(1) = \alpha) \,
\mathbf{E}_i^h\left[ \mathbf{1}_{\{\mathcal{T} \text{ finite}\}} s^{\mathrm{Leaftype}(\mathrm{\mathcal{T}})} \,\middle|\, Z(1) = \alpha \right].
\end{align}
Now, on the one hand, in the case $\alpha = 0$ we have 
\begin{align} \label{eq:O1}
\mathbf{E}_i^h\left[ \mathbf{1}_{\{\mathcal{T} \text{ finite}\}} s^{\mathrm{Leaftype}(\mathrm{\mathcal{T}})} \,\middle|\, Z(1) = 0 \right] = s_i.
\end{align}
On the other hand, according to the branching property, for $\alpha \neq 0$ we have
\begin{align} \label{eq:O2}
\mathbf{E}_i^h\left[ \mathbf{1}_{\{\mathcal{T} \text{ finite}\}} s^{\mathrm{Leaftype}(\mathrm{\mathcal{T}})} \,\middle|\, Z(1) = \alpha \right] = \prod_{j=1}^n E_j(s_1,\ldots,s_n)^{\alpha_j}  =: E(s_1,\ldots,s_n)^\alpha.
\end{align}
Using \eqref{eq:O1} and \eqref{eq:O2} in \eqref{eq:O0} we obtain
\begin{align} \label{eq:branching}
E_i(s_1,\ldots,s_n) = \sum_{\alpha \neq 0} h_{i,\alpha}E(s_1,\ldots,s_n)^\alpha + d_i s_i = h_i(E(s_1,\ldots,s_n)) + d_i s_i,
\end{align}
and thus, \eqref{eq:mahler} holds.
\end{proof}

\begin{proof}[Proof of Theorem \ref{thm:inversion}]
Let $F\colon\mathbb{C}^n \to \mathbb{C}^n$ be a map with components of the form \eqref{eq:components}.
Since \eqref{eq:as1} guarantees that the matrix $I-A$ is invertible (with inverse given by $\sum_{k=0}^\infty A^k$), the Jacobian matrix of $F$ at the origin
\begin{align*}
\left( \frac{\partial}{\partial X_j } F_i(X_1,\ldots,X_n) \right)\Big|_{(X_1,\ldots,X_n) = 0} = \delta_{i,j} - h_{i,\mathbf{e}_j} 
= (I-A)_{i,j}
\end{align*}
is invertible.
By the implicit function theorem (see e.g.\ \cite[Section 1.3]{taylor}), $F$ has an analytic inverse $F^{-1}\colon\mathbb{C}^n \to \mathbb{C}^n$ that satisfies $F^{-1}(0) = 0$ and converges in some neighbourhood of the origin.
Thus, we may write $F^{-1}_i(X_1,\ldots,X_n) = \sum_{\alpha \neq 0 } c_{i,\alpha} X^\alpha$, for certain complex coefficients $c_{i,\alpha}$.
We need to show that $F^{-1}=g$, i.e.\ that $c_{i,\alpha}=g_{i,\alpha}$ is the coefficient defined in~\eqref{eq:Gco}, for all $i$ and $\alpha$.

By considering the higher order derivatives of $I = F \circ F^{-1}$ (or, say, using the multivariate Fa\`a di Bruno formula \cite{JP}), one can see that the $X^\alpha$ coefficient $c_{i,\alpha}$ of $F^{-1}_i$ may be expressed as a polynomial in the coefficients of $F$ and $I-A$ of degree at most $|\alpha|$.
Notice also that, if $(h_{i,\alpha}\colon i \in [n], \alpha \in \mathbb{Z}_{\geq 0}^n \setminus \{0\})$ satisfy \eqref{eq:GWo}, then $A := (h_{i,\mathbf{e}_j})_{1 \leq i,j \leq n}$ is a strictly sub-stochastic matrix (i.e., it has nonnegative entries and row sums strictly less than $1$), and accordingly $I - A$ is invertible and \eqref{eq:as1} holds.
It follows that, in order to characterise $c_{i,\alpha}$ as a polynomial in the variables $(h_{j,\beta} \colon j \in [n], \beta\neq 0, |\beta| \leq |\alpha|)$, it is sufficient, by analytic continuation, to do so under the constraint \eqref{eq:GWo}.
Thus, for the remainder of the proof, we may assume \eqref{eq:GWo} holds.

Under the assumption \eqref{eq:GWo}, by \eqref{eq:mahler}, we have $E = hE+D$, i.e.\ $(I-h)E = D$.
By \eqref{eq:equalPowerSeries}, we also have $E= gD$, so that 
\begin{align*}
(I-h)gD = D.
\end{align*}
Since $d_i > 0$ for all $i$, we have that $D$ is invertible and we can compose on the right by $D^{-1}$ to obtain 
\begin{align*}
(I-h)g=I.
\end{align*}
This proves that $g$ is a {right} inverse of $F=I-h$.
Since the set of power series {mappings} in $n$ variables with invertible linear term is a group under composition, $g$ is also a left inverse, so that $g=F^{-1}$, as desired.
\end{proof}

\subsection{Inversion formula for unordered trees}
\label{subsec:inversionUnordered}

We now give an analogue of Theorem \ref{thm:inversion} for power series with the $\alpha!$ in the coefficients (see \eqref{eq:alphaFact}).
The modified formulas involve unordered trees (in $\mathbb{S}_{i,\alpha}$) rather than ordered trees (in $\ordS_{i,\alpha}$).

\begin{thm} \label{thm:inversion2}
 Let $F\colon\mathbb{C}^n \to \mathbb{C}^n$ be a power series mapping with components of the form
\begin{equation}
\label{eq:powerSeries}
F_i(X_1,\ldots,X_n) := X_i - \sum_{\alpha\neq 0}\frac{H_{i,\alpha}}{\alpha!} X^\alpha,
\end{equation}
which converge in a neighbourhood of $0$. 
Suppose that the series $\sum_{k=0}^{\infty} A^k$ converges, with $A_{i,j} = H_{i,\mathbf{e}_j}/j!$. 
Then, $F$ has an analytic inverse given by $F^{-1}_i(X_1,\ldots,X_n) := \sum_{\alpha\neq 0} g_{i,\alpha}X^\alpha$, where
\begin{align} \label{eq:june}
g_{i,\alpha} := \sum_{ \mathcal{T} \in \mathbb{S}_{i,\alpha}} \prod_{\substack{v \text{ internal} \\ \text{vertex of } \mathcal{T}}} \frac{H_{\tau(v),\mu(v)}}{|\mu(v)|!}.
\end{align}
Here, $|\mu(v)|$ is the total number of children of an internal vertex $v$.
\end{thm} 

\begin{proof}
Setting $h_{i,\alpha} := H_{i,\alpha}/\alpha!$ and using Theorem \ref{thm:inversion}, we see that the coefficients of $F^{-1}_i$ are
\begin{align*}
g_{i,\alpha} = \sum_{\mathcal{T} \in \ordS_{i,\alpha} } \prod_{\substack{v \text{ internal} \\ \text{vertex of } \mathcal{T}}} \frac{H_{\tau(v),\mu(v)}}{\mu(v)!}. 
\end{align*}
Note now that, given a vertex $v$ with $\mu(v)$ children, there are $\mu(v)!$ unordered ways of ordering these children, and $|\mu(v)|!$ ordered ways.
Replacing these factors yields the result.                                     
\end{proof}

Note that the inversion formula \eqref{eq:Gsum} is a consequence of Theorem \ref{thm:inversion2}. Indeed, if $H_{i,\alpha}$ is nonzero only when $|\alpha|=d$, it follows that from \eqref{eq:june} that we may write
\begin{align} \label{eq:june}
g_{i,\alpha} := \frac{1}{(d!)^k} \sum_{ \mathcal{T} \in \mathbb{S}_{i,\alpha}, \text{ $d$-ary}} \prod_{\substack{v \text{ internal} \\ \text{vertex of } \mathcal{T}}} H_{\tau(v),\mu(v)}.
\end{align}
Now note that a sum over all $d$-ary unordered labelled trees $\mathcal{T}$ in $\mathbb{S}_{i,\alpha}$ 
is equivalent to a sum over all $(i,\alpha)$ labellings of a $d$-Catalan tree. Thus we obtain \eqref{eq:Gsum}.

\section{Keller property and combinatorial nilpotency} \label{sec:Keller}

In this section we provide a more complete discussion of the combinatorial approach to the Jacobian conjecture, as outlined in Section \ref{sec:combapproach}.

Recall from Section \ref{subsec:jacobianIntro} that a Keller mapping is a polynomial mapping whose Jacobian determinant equals a non-zero complex constant. As mentioned in the introduction,
Theorem~\ref{thm:inversion2} conveys a combinatorial approach to the Jacobian conjecture: in order to prove the Jacobian conjecture, it is sufficient to verify that whenever $F={I}-H$ is a Keller mapping with components of the form~\eqref{eq:powerSeries}, the inverse coefficients $g_{i,\alpha}$ defined in \eqref{eq:june} are zero for all but finitely many $\alpha$. 

At this stage we may reduce our combinatorial task substantially further by appealing to the celebrated Bass--Connell--Wright \cite{BCW} reduction of the Jacobian conjecture. Let the \textbf{degree} 
of a polynomial mapping be the maximum of the degrees of its components.
A polynomial mapping is said to be \textbf{homogeneous} of degree $d$ if each of its components is homogeneous of degree $d$.

\begin{thm}[Bass, Connell and Wright \cite{BCW}]
\label{thm:BassConnellWright}
Suppose there exists $d \geq 3$ with the following property:
for all $n$, every polynomial mapping of the form $F = {I} -H\colon\mathbb{C}^n \to \mathbb{C}^n$ with $H$ homogeneous of degree $d$ and $JH$ nilpotent, is a polynomial automorphism.
Then the Jacobian conjecture holds in full generality.
\end{thm}

We may then restrict from now on to polynomial mappings $F=I-H$, where
\begin{equation}
\label{eq:homogeneousMap}
H\colon \mathbb{C}\to \mathbb{C}, \qquad
H_i(X_1,\ldots,X_n) = \sum_{ |\alpha| = d } \frac{H_{i,\alpha}}{\alpha!} X^\alpha \qquad \text{for all } i=1,\dots, n.
\end{equation}
We canonically associate $F$ (or $H$) with the coefficients $H = (H_{i,\alpha}\colon i \in [n], |\alpha| =d)$.
Theorem \ref{thm:BassConnellWright} also tells us that we should focus on mappings of this form with $JH$ nilpotent. 
In fact, for such mappings, the condition that $JH$ is nilpotent is equivalent to Keller's condition that $\det(JF) = 1$ (see \cite[Section 2.5]{JP}).

In this setting, we now work towards giving an even more explicit combinatorial statement of the Jacobian conjecture, with the Keller property (in this case appearing as the nilpotency of $JH$) also put in combinatorial terms, and more precisely in terms of shuffle classes of $d$-Catalan trees.
Strictly speaking, such a combinatorial interpretation is not new, and variations of it have appeared in numerous places \cite{ JP,singer1,singer2, wright05a}.
In particular, we generalise from $\mathcal{C}_k^{(2)}$ to $\mathcal{C}_k^{(d)}$ tools developed by Singer in \cite{singer1,singer2}, providing concise proofs.

\subsection{Shuffling $d$-Catalan trees} \label{subsec:shuffles}


{We begin by defining shuffle classes.}
Roughly speaking, a length-$p$ shuffle class is a collection of trees that may be obtained from one another by shuffling the subtrees coming off a fixed length-$p$ path of descendant vertices.

Recall from the introduction that, for $k \geq 1$, $\mathcal{C}_k^{(d)}$ denotes the set of $d$-Catalan trees {with $k$ internal vertices; such trees have $dk+1$ vertices in total, and hence $(d-1)k+1$ leaves.}
An \textbf{(ancestral) path} of length $p$ in a $d$-Catalan tree is a sequence of vertices $(v_0,\ldots,v_p)$ such that each $v_i$ is a child of $v_{i-1}$.
The \textbf{height} of a vertex $v$ is its graph distance from the root.
Each vertex $v$ of height at least $p$ is associated to a unique path $(v_0,\ldots,v_p)$ of length $p$ ending with $v = v_p$.

\begin{df}[Shuffle class] \label{df:shuffle}
{Let $T\in\mathcal{C}_k^{(d)}$.}
Let $p \geq 1$ be an integer and assume that the vertex $v$ of $T$ has height at least $p$.
Let $(v_0,\ldots,v_p)$ be the unique path associated with $v = v_p$.
For each $i \geq 1$, $v_i$ has $d-1$ siblings $w_{i,1},\ldots,w_{i,d-1}$; each of these siblings $w_{i,j}$ subtends a subtree $W_{i,j}$, which is itself a $d$-Catalan tree with fewer leaves than $T$.
The \textbf{shuffle class} {$\mathrm{Sh}(T;v_0,\ldots,v_p)$} of the path $(v_0,\ldots,v_p)$ is the subset of $\mathcal{C}_k^{(d)}$ consisting of all Catalan trees $T'$ that may be obtained from $T$ by {reshuffling the $(d-1)p$ subtrees $W_{i,j}$ across all $p$ generations}.
See Figure \ref{fig:shuffle} for an illustration.
\end{df}

Note that a shuffle class $\mathrm{Sh}(T; v_0,\ldots,v_p)$ may contain up to $((d-1)p)!$ different trees in $\mathcal{C}_k^{(d)}$.
However, if some of the subtrees $W_{i,j}$ are the same, then the shuffle class will contain {less than $((d-1)p)!$} {elements}.
It is even possible that all the subtrees $W_{i,j}$ are identical; in this case, the shuffle class will consist of exactly one tree.

If a shuffle class contains more than one element, it can be represented in different ways.
Indeed, if $T' \in \mathrm{Sh}(T; v_0,\ldots,v_p)$, then there exists a path $(v_0',\ldots,v_p')$ of vertices in $T'$ such that $\mathrm{Sh}(T; v_0,\ldots,v_p) = \mathrm{Sh}(T'; v'_0,\ldots,v'_p)$.

Recall from Section \ref{subsec:invFormula} that, for $i \in [n]$ and $\alpha \in \mathbb{Z}_{ \geq 0}^n$, $\mathbb{S}_{i,\alpha}$ is the set of (unordered) $n$-type trees $\mathcal{T}=(T,\tau)$ with root type $i$ and $\alpha_j$ leaves of type $j$.
We equivalently say that $\mathcal{T}$ is an \textbf{$(i,\alpha)$ labelling} of $T$.
Recall also that, for any internal vertex $v$, $\mu(v)$ is the multi-index whose $j^{\text{th}}$ component counts the number of type $j$ children of $v$.

Suppose now $T\in\mathcal{C}^{(d)}_k$, so that $|\mu(v)|=d$ for all internal vertex $v$ of $T$, and the total number of leaves is $|\alpha|=(d-1)k+1$.
Let also $H = (H_{i,\alpha}\colon i \in [n], \alpha \in \mathbb{Z}_{\geq 0}^n , |\alpha| = d)$ be a collection of coefficients in $\mathbb{C}$.
We define the \textbf{$H$-weight} $\mathcal{E}_H(\mathcal{T})$ of the labelling $\mathcal{T}$ of $T$ as
\begin{align} \label{eq:blender}
\mathcal{E}_H(\mathcal{T}) := \prod_{\substack{v \text{ internal} \\ \text{vertex of } \mathcal{T}}} H_{\tau(v),\mu(v)}.
\end{align}
\begin{figure}
  \centering   
\scalebox{.8}{
 \begin{forest}
[1, fill=c1
[1,  fill=c1
[2, fill=c2]
[4,fill=c4]
[1,fill=c1]
]
[2, fill=c2
[2,fill=c2]
[3,  fill=c3
[3, fill=c3]
[1,  fill=c1]
[2, fill=c2]
]
[3,  fill=c3]
]
[4, fill=c4
]
]
\end{forest}
}  
\captionsetup{width=.85\linewidth}
\caption{An $(i,\alpha)$ labelling $\mathcal{T}$ of a tree $T\in\mathcal{C}_k^{(d)}$, with $d=3$, $n=4$, $i=1$, $k=4$, and $\alpha=(2,3,2,2)$.
The $H$-weight of the depicted labelled tree $\mathcal{T}$ is $\mathcal{E}_H(\mathcal{T}) = H_{1,(1,1,0,1)}^2 H_{2,(0,1,2,0)} H_{3,(1,1,1,0)}$.}\label{fig:mbt1}
\end{figure}
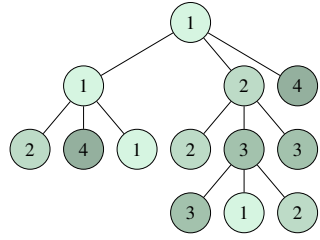
See Figure \ref{fig:mbt1} for an example.
We also define the \textbf{average $H$-weight} of $T$ as
\begin{align} \label{eq:F}
E_{i,\alpha,H}(T) := \sum_{ \text{$(i,\alpha)$ labellings $\mathcal{T}$ of $T$}} \mathcal{E}_H(\mathcal{T}).
\end{align}

The value of defining shuffle classes lies in the following lemma, which states that the nilpotency of $JH$ guarantees that sums of average $H$-weights over shuffle classes are zero.

\begin{lemma}[Shuffle lemma] \label{lem:nilp}
Let $H$ be of the form~\eqref{eq:homogeneousMap}. 
Suppose that the Jacobian matrix $JH$ of $H$ satisfies $(JH)^p = 0$ for some integer $1 \leq p \leq n$.
Then
\begin{align} \label{eq:Nsum} 
\sum_{ T \in \mathcal{S} }  E_{i,\alpha,H}(T)=0
\qquad\text{for all } i\in[n] \text{ and } |\alpha|=(d-1)k+1
\end{align}
and for every shuffle class $\mathcal{S}$ of $\mathcal{C}_k^{(d)}$ associated with a path of length $p$.
\end{lemma}
In Lemma \ref{lem:nilp} we have isolated the so-called index of nilpotency $p$ from the underlying dimension $n$ of the polynomial mapping $H$.
Of course, if $JH$ is nilpotent, then $(JH)^n=0$ and we may set $p = n$.

We note again that Lemma \ref{lem:nilp} is not strictly new: the general idea is present in the works of Singer \cite{singer1,singer2}, {where an analogous result is proven in the $d=2$ case.}
We will prove Lemma \ref{lem:nilp} in Section \ref{subsec:combNilp}.

\subsection{Algebraic nilpotency}
\label{subsec:algebraicNilp}

Our first step is to express the nilpotency of $JH$ in terms of explicit algebraic equations in the coefficients $H_{i,\alpha}$.

\begin{lemma} \label{lem:kelp}
Let $H$ be of the form~\eqref{eq:homogeneousMap}. 
Then, for any integer $1 \leq p \leq n$, we have $(JH)^p= 0$ if and only if 
\begin{align*}
\sum_{ \beta^1 + \cdots + \beta^p = \beta} {\left(\prod_{\ell = 1}^p \frac{ (d-1)!}{ \beta^\ell!}\right)} \sum_{ k_1,\ldots,k_{p-1} \in [n]} \prod_{\ell = 1}^p  H_{k_{\ell-1},\beta^\ell + \mathbf{e}_{k_\ell}} = 0
\end{align*}
for all $i,j \in [n]$ and $|\beta| = (d-1)p$, where {$k_0:=i$, $k_p:=j$}, and the outer sum is over all $p$-tuples $(\beta^1,\ldots,\beta^p)$ of multi-indices $\beta^\ell$ in $\mathbb{Z}_{\geq 0}^n$ with $|\beta^\ell| = d-1$ for all $\ell$, whose sum is the multi-index $\beta$. 
\end{lemma}

\begin{proof}
The $(i,j)^{\text{th}}$ entry of $JH=(\frac{\partial}{\partial X_j} H_i)_{1\leq i,j\leq n}$ is given by the polynomial
\begin{align*}
(JH)_{i,j}(X_1,\ldots,X_n ) = \sum_{|\alpha| = d \colon \alpha_j \geq 1 } \frac{ H_{i,\alpha} \alpha_j }{\alpha!} X^{\alpha - \mathbf{e}_j} = \sum_{ |\beta| = d-1 } \frac{ H_{i,\beta + \mathbf{e}_j}}{ \beta!} X^\beta.
\end{align*}
The $(i,j)^{\text{th}}$ entry of $(JH)^p$ is the polynomial 
\begin{align} \label{eq:kalpa}
((JH)^p)_{i,j}(X_1,\ldots,X_n ):= \sum_{ |\beta^1|,\ldots,|\beta^p|=d-1} \sum_{ k_1,\ldots,k_{p-1} \in [n]} {\prod_{\ell = 1}^p \frac{ H_{k_{\ell-1},\beta^\ell + \mathbf{e}_{k_\ell}}}{ \beta^\ell!} X^{\beta^\ell},}
\end{align}
{where $k_0 := i$, $k_p:= j$, and the outer sum is over all $p$-tuples of multi-indices $\beta^\ell$ in $\mathbb{Z}_{\geq 0}^n$ with $|\beta^\ell| = d-1$ for all $\ell$.}

{Now, we have $(JH)^p = 0$ if and only if \eqref{eq:kalpa} is zero for all $i,j$.
Since $(JH)^p$ is clearly a homogeneous polynomial of degree $(d-1)p$, we have $(JH)^p=0$ if only if  the coefficient of $X^\beta$ in $((JH)^p)_{i,j}(X_1,\ldots,X_n )$ is zero for all $i,j$ and for each $|\beta| = (d-1)p$.
By \eqref{eq:kalpa}, such a coefficient is given by}
\begin{align*}
\sum_{ \beta^1 + \cdots + \beta^p = \beta} \sum_{ k_1,\ldots,k_{p-1} \in [n]} \prod_{\ell = 1}^p \frac{ H_{k_{\ell-1},\beta^\ell + \mathbf{e}_{k_\ell}}}{ \beta^\ell!},
\end{align*}
where the outer sum is over all $p$-tuples $(\beta^1,\ldots,\beta^p)$ of multi-indices $\beta^\ell$ in $\mathbb{Z}_{\geq 0}^n$ with $|\beta^\ell| = d-1$ for all $\ell$, whose sum is $\beta$.
The result then follows by multiplying this coefficient by $((d-1)!)^p$.
\end{proof}

\subsection{Combinatorial nilpotency: labelled version}
\label{subsec:combNilp}


From Lemma \ref{lem:kelp}, we now deduce a `labelled version' of the shuffle lemma (Lemma \ref{lem:nilp}).
We thus begin by defining the notion of labelled shuffle classes.
\begin{df} \label{df:shuffle2}
Let $\mathcal{T}=(T,\tau)$ be an $n$-type tree, where $T\in \mathcal{C}_k^{(d)}$.
Let $(v_0,\ldots,v_p)$ be any length-$p$ path in $\mathcal{T}$.
The \textbf{labelled shuffle class} $\mathrm{Sh}(\mathcal{T};v_0,\ldots,v_p)$ is the set of $n$-type trees $\mathcal{T}'$ that may be obtained from $\mathcal{T}$ by:
\begin{itemize}
\item rearranging the {labelled} subtrees subtended by the $(d-1)p$ siblings of $v_1,\ldots,v_p$;
\item giving arbitrary new types in $\{1,\ldots,n\}$ to the vertices $v_1,\ldots,v_{p-1}$.
\end{itemize}
\end{df}

A labelled shuffle class $\mathrm{Sh}(\mathcal{T};v_0,\ldots,v_p)$ of length $p$ may contain up to $n^{p-1} ((d-1)p)!$ different labelled trees.
If some of the labelled subtrees subtended by the siblings are identical (in both their topology and labelling), then the respective labelled shuffle class will contain fewer elements.

\begin{lemma} \label{lem:nilp2}
Let $H$ be of the form~\eqref{eq:homogeneousMap}.
Suppose that the Jacobian matrix $JH$ of $H$ satisfies $(JH)^p = 0$ for some integer $1 \leq p \leq n$.
Then, for all $n$-type tree $\mathcal{T}=(T,\tau)$, with $T\in \mathcal{C}_k^{(d)}$, and every path $(v_0,\ldots,v_p)$ in $\mathcal{T}$, we have
\begin{align*}
\sum_{ \mathcal{T}' \in \mathrm{Sh}(\mathcal{T};v_0,\ldots,v_p) } \mathcal{E}_H(\mathcal{T}') = 0.
\end{align*}

\end{lemma}

\begin{proof}
We can decompose the weight of $\mathcal{T}$ as 
\begin{align*}
\mathcal{E}_H(\mathcal{T}) = \prod_{v \text{ internal}} H_{\tau(v),\mu(v)} = \mathcal{E}_H(\mathcal{R}) \mathcal{E}_H(\mathcal{S}) \left(\prod_{i=1}^p \prod_{j=1}^{d-1} {\mathcal{E}_H(\mathcal{W}_{i,j})}  \right) {\mathcal{E}_H(\mathcal{U})},
\end{align*}
where
\begin{itemize}
\item $\mathcal{R}$ is the labelled subtree obtained from $\mathcal{T}$ by removing any vertex that is a descendant of $v_0$ (but keeping $v_0$ itself);
\item $\mathcal{S}$ is the labelled subtree given by $v_p$ and its descendants;
\item
$\mathcal{W}_{i,j}$ is the labelled subtree given by $w_{i,j}$ and its descendants, where, for each $1 \leq i \leq p$, $w_{i,1},\ldots,w_{i,d-1}$ are the $d-1$ siblings of $v_i$;
\item $\mathcal{U}$ is the labelled subtree consisting of the vertices $v_0,\ldots,v_p$ and the siblings of $v_1,\ldots,v_p$.
\end{itemize}
As a shorthand we write
\begin{align*}
G_H(\mathcal{T}) :=  \mathcal{E}_H(\mathcal{R}) \mathcal{E}_H(\mathcal{S}) \prod_{i=1}^p \prod_{j=1}^{d-1} {\mathcal{E}_H(\mathcal{W}_{i,j})}.
\end{align*}
 
Notice now that any $\mathcal{T}'\in \mathrm{Sh}(\mathcal{T};v_0,\dots,v_p)$ contains the trees $\mathcal{R}$, $\mathcal{S}$, and each of the $\mathcal{W}_{i,j}$.
Thus, its weight differs from $\mathcal{T}$ only in the labelling of $v_1,\ldots,v_{p-1}$ and the rearrangement of the siblings.
Therefore,
\begin{align*}
\mathcal{E}_H(\mathcal{T}') = G_H(\mathcal{T}) \mathcal{E}_H(\mathcal{U}'),
\end{align*}
where $\mathcal{U}'$ is the labelled subtree of $\mathcal{T}'$ consisting of the vertices $v_0,\ldots,v_p$ and the siblings of $v_1,\ldots,v_p$.
In particular,
\begin{align} \label{eq:kohl}
\sum_{ \mathcal{T}' \in \mathrm{Sh}(\mathcal{T};v_0,\ldots,v_p) } {\mathcal{E}_H(\mathcal{T}')} =G_H(\mathcal{T})  \sum_{\mathcal{U}'} \mathcal{E}_H(\mathcal{U}'),
\end{align}
where the sum is over all labelled trees $\mathcal{U}'$ that can be obtained from $\mathcal{U}$ by relabelling $v_1,\ldots,v_{p-1}$ and by rearranging the positions of its leaves {(except $v_p$)}.
Suppose that the total leaf type of $\mathcal{U}$ is $\beta$, i.e.\ the total number of leaves of type $k$ in $\mathcal{U}$ is the $k^{\text{th}}$ component of $\beta$; then, the total leaf type of each $\mathcal{U}'$ is also $\beta$.

Note that each such $\mathcal{U}'$ has {(potentially)} different labels $k_1,\ldots,k_{p-1} \in [n]$ for the vertices $v_1,\ldots,v_{p-1}$.
Note also that, given any $p$-tuple $(\beta^1,\ldots,\beta^p)$ of multi-indices of degree $d-1$ adding to $\beta$, and any $k_1,\ldots,k_{p-1}$ in $[n]$, there are $\prod_{\ell=1}^p \frac{{(d-1)!}}{\beta^\ell!}$ different $\mathcal{U}'$ such that $\tau(v_\ell) = k_\ell$ for $1 \leq \ell \leq p-1$ and {such that the types of the leaves in generation $\ell$ (excluding $v_p$) are encoded by the multi-index $\beta^\ell$}; this follows from the fact that there are $(d-1)!/\beta^\ell!$ different planar ways of rearranging the order of the $d-1$ leaves in generation $\ell$ such that, for each $k$, $\beta^\ell_k$ of them have type $k$.
In short,
\begin{align} \label{eq:legall}
 \sum_{\mathcal{U}'} \mathcal{E}_H(\mathcal{U}') = \sum_{ \beta^1 + \cdots + \beta^p = \beta} {\left(\prod_{\ell = 1}^p \frac{ (d-1)!}{ \beta^\ell!}\right)} \sum_{ k_1,\ldots,k_{p-1} \in [n]} \prod_{\ell = 1}^p  H_{k_{\ell-1},\beta^\ell + \mathbf{e}_{k_\ell}},
\end{align}
{where $k_0:=\tau(v_0)$ and $k_p:=\tau(v_p)$ are fixed.}
Now plug \eqref{eq:legall} into \eqref{eq:kohl}, and use Lemma \ref{lem:kelp} to conclude.
\end{proof}

The latter proof suggests that we may specialise Lemma \ref{lem:nilp2} to a particular kind of trees, which we call $d$-Catalan ferns.
We define a \textbf{$d$-Catalan fern} to be an element of $\mathcal{C}_p^{(d)}$ of height $p$, with a designated `sink' vertex in generation $p$, and the property that at most one vertex in each generation is not a leaf.
We define an \textbf{$(i,j,\alpha)$ labelling} of a $d$-Catalan fern $U$ to be any labelling of $U$ giving the root type $i$, the sink type $j$, and giving exactly $\alpha_\ell$ of the non-sink leaves type $\ell$, for all $\ell\in[n]$. See Figure \ref{fig:mbt2} for an illustration.

By taking $\mathcal{T}$ to be a labelled $d$-Catalan fern of height $p$ in Lemma \ref{lem:nilp2}, we readily obtain the following result:
\begin{cor}[Lemma 2.8 of \cite{JP}] \label{coro:nilp2}
Let $H$ be of the form~\eqref{eq:homogeneousMap}.
Suppose that the Jacobian matrix $JH$ of $H$ satisfies $(JH)^p = 0$ for some integer $1 \leq p \leq n$.
Then, for any $d$-Catalan fern $U$ of height $p$, we have
\begin{equation}
\label{eq:sumFerns}
\sum_{ \text{$(i,j,\alpha)$ labellings $\mathcal{U}$ of $U$}} \mathcal{E}_H(\mathcal{U}) =0 .
\end{equation}
\end{cor}
Corollary \ref{coro:nilp2} is precisely the third bullet point of the combinatorial approach to the Jacobian conjecture stated in Section \ref{sec:combapproach}.
\subsection{Combinatorial nilpotency: unlabelled version}
\label{subsec:combNilp_unlabelled}

We can now use Lemma \ref{lem:nilp2} to prove Lemma \ref{lem:nilp}.

\begin{proof}[Proof of Lemma \ref{lem:nilp}]
Let $\mathrm{Sh}(T; v_0,\dots,v_p)$ be a shuffle class associated with a path of length $p$, for some fixed $T\in \mathcal{C}^{(d)}_k$.
Let $i\in[n]$ and $|\alpha|=k(d-1)+1$.
We need to show that
\begin{equation}
\label{eq:shuffleLemma_proof}
\sum_{ T' \in \mathrm{Sh}(T; v_0,\dots,v_p) }  \; \sum_{ \substack{\text{$(i,\alpha)$ labellings} \\ \text{$\mathcal{T}'$ of $T'$}}} \mathcal{E}_H(\mathcal{T}')=0 .
\end{equation}
The labelled tree $\mathcal{T}'$ in the expression above is obtained as follows:
\begin{enumerate}[label=(\alph*)]
\item choose $T'\in \mathrm{Sh}(T; v_0,\dots,v_p)$;
\item given such $T'$, choose an $(i,\alpha)$ labelling $\mathcal{T}'$ of $T'$.
\end{enumerate}
It is clear that any such $\mathcal{T}'$ can be obtained in only one way via the above procedure.

Consider now the alternative procedure:
\begin{enumerate}[label=(\alph*')]
\item\label{item:altern1} choose an $(i,\alpha)$ labelling $\mathcal{T}$ of $T$;
\item\label{item:altern2} given such $\mathcal{T}$, choose $\mathcal{T}'\in \mathrm{Sh}(\mathcal{T}; v_0,\dots,v_p)$.
\end{enumerate}
In general, there are several different ways of obtaining any $\mathcal{T}'$ via the alternative procedure.
We then have
\begin{equation}
\label{eq:shuffleLemma_proof2}
\sum_{ T' \in \mathrm{Sh}(T; v_0,\dots,v_p) }  \; \sum_{ \substack{\text{$(i,\alpha)$ labellings} \\ \text{$\mathcal{T}'$ of $T'$}}} \mathcal{E}_H(\mathcal{T}')
= \sum_{ \substack{\text{$(i,\alpha)$ labellings} \\ \text{$\mathcal{T}$ of $T$}}} \; \sum_{ \mathcal{T}' \in \mathrm{Sh}(\mathcal{T}; v_0,\dots,v_p) } \frac{\mathcal{E}_H(\mathcal{T}')}{N(\mathcal{T}')} ,
\end{equation}
where $N(\mathcal{T}')\geq 1$ is the number of ways of obtaining $\mathcal{T}'$ via the second procedure.
Let us now analyse such a number.

For $1\leq i\leq p$, denote by $\mathcal{W}'_{i,j}$, $1\leq j\leq d-1$, the labelled trees subtended by the $d-1$ siblings of $v_i$ in $\mathcal{T}'$.
The greater flexibility provided by the second procedure amounts to two facts: first, in the step \ref{item:altern1} one may choose to label the vertices $v_1,\dots,v_{p-1}$ \emph{arbitrarily}, since these can then be given new types in the labelled tree shuffling of step \ref{item:altern2}; second, if in $\mathcal{T}'$ some of the $\mathcal{W}'_{i,j}$ have identical tree structure but different labellings, then there is more than one way to label such trees in step \ref{item:altern1} so that an appropriate reshuffling in step \ref{item:altern2} yields exactly $\mathcal{T}'$.
More precisely, we have
\[
N(\mathcal{T}') = n^{p-1} \prod_{W \text{ planar tree}} N_W(\mathcal{T}') ,
\]
where $N_W(\mathcal{T}')$ is the multinomial coefficient counting the number of \emph{distinct} ways of rearranging all the $\mathcal{W}'_{i,j}$ that have tree structure $W$ (but potentially different labellings); by convention, $N_W(\mathcal{T}')$ is set to be $1$ when no $\mathcal{W}'_{i,j}$ has tree structure $W$.
Alternatively, we have the formula
\[
N(\mathcal{T}') = \frac{\#\mathrm{Sh}(\mathcal{T}'; v_0,\dots,v_p)}{\#\mathrm{Sh}(T; v_0,\dots,v_p)} ,
\]
where we stress that the numerator is the cardinality of a labelled shuffle class, while the denominator is the cardinality of an unlabelled shuffle class.
It is clear from both the above formulas that $N(\mathcal{T}')$ is constant on all trees $\mathcal{T}'$ in $\mathrm{Sh}(\mathcal{T}; v_0,\dots,v_p)$; equivalently, $N(\mathcal{T}')=N(\mathcal{T})$.
Therefore, \eqref{eq:shuffleLemma_proof2} reads
\[
\sum_{ T' \in \mathrm{Sh}(T; v_0,\dots,v_p) }  \; \sum_{ \substack{\text{$(i,\alpha)$ labellings} \\ \text{$\mathcal{T}'$ of $T'$}}} \mathcal{E}_H(\mathcal{T}')
= \sum_{ \substack{\text{$(i,\alpha)$ labellings} \\ \text{$\mathcal{T}$ of $T$}}} \frac{1}{N(\mathcal{T})} \sum_{ \mathcal{T}' \in \mathrm{Sh}(\mathcal{T}; v_0,\dots,v_p) } \mathcal{E}_H(\mathcal{T}').
\]
We now use Lemma \ref{lem:nilp2} to conclude the proof of \eqref{eq:shuffleLemma_proof}.
\end{proof}

\section{The shuffle conjectures} \label{sec:conjectures}


\subsection{Subtree shuffling conjecture}
\label{subsec:orthogonality}

In this section we prove Theorem \ref{thm:relation} and Theorem \ref{thm:shufflechain}, which state that certain conjectures involving shuffling subtrees of $d$-Catalan trees imply the Jacobian conjecture. While our work so far in Sections \ref{sec:inversion} and \ref{sec:Keller} has been based on reshaping (and finding new approaches to) known results in the literature, in this section we introduce our new and foundational idea based on the orthogonality of indicators of shuffle classes to the average energy function. This converts the problem from one involving weighted trees to a simpler one involving unweighted trees.

The basic idea, which lies at the heart of our approach, is the following.
Recall first that, to prove the Jacobian conjecture, one would need to show that, if $JH$ is nilpotent, the quantity $\Phi_i^\alpha(H)$ defined in \eqref{eq:Gsum} is zero for large $|\alpha|$.
On the other hand, as we showed in Lemma \ref{lem:nilp}, the nilpotency condition implies that variants of the sum occurring in \eqref{eq:Gsum}, with the $T$-sum over a shuffle class $\mathcal{S}\subseteq \mathcal{C}_k^{(d)}$ rather than the whole $\mathcal{C}_k^{(d)}$, are zero.
Putting the two observations together, we conclude: \emph{if} the constant function can be written as a linear combination of indicators of shuffle classes $\mathcal{S}$, then $\Phi_i^\alpha(H)=0$, as desired.

To make the above argument rigorous, let us define an inner product on the set of functions $(\mathcal{C}_k^{(d)})^* := \{ f: \mathcal{C}_k^{(d)} \to \mathbb{C} \}$ by setting
\begin{align*}
\langle f, g \rangle := \sum_{T \in \mathcal{C}_k^{(d)} } \overline{f(T)} g(T).
\end{align*} 
Given integers $k \geq 1$ and $d \geq 2$, the average $H$-weight function $E_{i,\alpha,H}$, defined in \eqref{eq:F}, is an element of $(\mathcal{C}_k^{(d)})^*$.
With this at hand, \eqref{eq:Gsum} reads
\begin{align} \label{eq:Gsum3}
\Phi_i^\alpha(H) = \frac{1}{(d!)^k}  \langle \mathbf{1}, E_{i,\alpha,H} \rangle,
\end{align}
where $\mathbf{1} \in(\mathcal{C}_k^{(d)})^*$ is the constant function defined by $\mathbf{1}(T) := 1$ for every $T$. 

Consider now the indicator function of a shuffle class $\mathcal{S}$ of length $p$, i.e.\ the element $\mathbf{1}_\mathcal{S}$ of $(\mathcal{C}_k^{(d)})^*$ defined by
\begin{align}
\mathbf{1}_\mathcal{S}(T) := 
\begin{cases}
1 \qquad &\text{if $T\in\mathcal{S}$,}\\
0 \qquad &\text{if $T\notin\mathcal{S}$}.
\end{cases}
\end{align}
The (shuffle) Lemma~\ref{lem:nilp} now reads
\begin{align} \label{eq:kalpa0}
(JH)^p = 0 \implies \langle \mathbf{1}_\mathcal{S} , E_{i,\alpha,H} \rangle = 0 \text{ for every length-$p$ shuffle class $\mathcal{S}\subseteq\mathcal{C}_k^{(d)}$.}
\end{align}

We are now able to show that the subtree shuffling conjecture posed in the introduction implies the Jacobian conjecture.
\begin{proof}[Proof of Theorem \ref{thm:relation}]
Fix $d \geq 3$.
Suppose that, for all $p \geq 1$, there exists some $k_0(d,p)$ such that, for all $k \geq k_0(d,p)$, $\mathbf{1} = \sum_{ \mathcal{S} } \lambda_S \mathbf{1}_S$ can be written as a linear combination of the length-$p$ shuffle indicators, for some constants
\[
\{ \lambda_\mathcal{S}\colon \mathcal{S} \text{ is a length-$p$ shuffle class in $\mathcal{C}_k^{(d)}$} \} .
\]
Then, using the conjugate linearity of $\langle \cdot, \cdot \rangle$ in the first slot, we have 
\begin{align*}
\langle \mathbf{1}, E_{i,\alpha,H} \rangle = \sum_{ \mathcal{S} } {\overline{\lambda_\mathcal{S}}} \langle \mathbf{1}_\mathcal{S} , E_{i,\alpha,H} \rangle .
\end{align*}
If $(JH)^p = 0$, then, according to \eqref{eq:kalpa0}, we have $\langle \mathbf{1}_\mathcal{S} , E_{i,\alpha,H} \rangle = 0$ for every length-$p$ shuffle class $\mathcal{S}$.
It follows from \eqref{eq:Gsum3} that $\Phi_i^\alpha(H) = 0$ for all $i\in [n]$ and {$|\alpha| = (d-1)k+1$}. 

Set now $p= n$.
In light of \eqref{eq:Gsum}, we have that, whenever $H$ is a homogeneous degree $d$ polynomial mapping s.t.\ $JH$ is nilpotent (i.e.\ with $(JH)^n = 0$), the coefficient $g_{i,\alpha}$ of $X^\alpha$ in the $i^{\text{th}}$ component of the inverse $F^{-1}$ of $F = {I} - H$ is zero, for all $\alpha$ such that {$|\alpha| \geq (d-1)k_0(d,n)+1$}.
In particular, $F^{-1}$ is a polynomial mapping.
{By the Bass--Connell--Wright reduction (Theorem \ref{thm:BassConnellWright}), this implies that the Jacobian conjecture is true.}
\end{proof}

{While Theorem \ref{thm:relation} states that Conjecture \ref{conj:shuffle} implies the Jacobian conjecture, we are not sure about the reverse implication, i.e.\ whether or not the Jacobian conjecture implies Conjecture \ref{conj:shuffle}.}
Nonetheless, we believe that Conjecture \ref{conj:shuffle} is a valuable line of investigation, since it represents a dramatic simplification of the combinatorial formulation of the Jacobian conjecture set out in Conjecture \ref{conj:JC2}: the latter involves complicated weighted sums over labellings of trees, whereas the former deals only with unweighted trees.

Furthermore, as we have already stated in the introduction and will prove in the next sections, it turns out that an approximate version of the ({subtree shuffling}) Conjecture~\ref{conj:shuffle} is true.

To conclude our discussion in this section, it is worth comparing {the subtree} shuffling conjecture with the following conjecture, posed by Singer in the $d=2$ setting.
\begin{conj}[Singer \cite{singer1}] \label{conj:strongsinger}
{Let $d \geq 2$, $p\geq 1$.
Then there exists $k_1(d,p)$ such that, whenever $k \geq k_1(d,p)$, the indicator functions of the length-$p$ shuffle classes in $\mathcal{C}_k^{(d)}$ span the vector space of functions $\mathcal{C}_k^{(d)} \to \mathbb{C}$. }
\end{conj}
Singer arrived at the latter while attempting to provide a combinatorial derivation of (Wang's) Theorem \ref{thm:wang}.
He used an innovative argument to prove Conjecture \ref{conj:strongsinger} in the cases $p=1,2,3$.
In the case $p = 3$, for instance, he proved that taking $k \geq 7$ is sufficient, thereby deducing that $F = {I} - H$ with $H$ quadratic and $(JH)^3 = 0$ has an inverse of degree at most $6$, \emph{regardless of the underlying dimension $n$}.

We remark that the $d=2$ case of the ({subtree shuffling}) Conjecture \ref{conj:shuffle} is substantially weaker than (Singer's) Conjecture \ref{conj:strongsinger}.
Conjecture \ref{conj:shuffle} states that the constant function lies in the span of the shuffle indicators, while Conjecture \ref{conj:strongsinger} states that \emph{all} functions lie in the span of the shuffle indicators.

\subsection{Shuffle chain conjecture} \label{subsec:markov}

We now discuss how to reformulate the machinery of Section \ref{subsec:orthogonality} in terms of Markov chains on trees, with the final goal of proving Theorem \ref{thm:shufflechain} from the introduction.
We believe this formulation potentially offers a tractable path to proving Conjecture \ref{conj:shuffle}, which, as we saw in Section \ref{subsec:orthogonality}, implies the Jacobian conjecture. 

In this section we consider a Markov chain that takes values in $\mathcal{C}_k^{(d)}$, and updates its state by randomly shuffling the subtrees along a randomly chosen path.

To define this process explicitly, recall that, for each vertex $v$ of height at least $p$ in a $d$-Catalan tree, there is a unique path $(v_0,\ldots,v_p)$ of length $p$ leading to $v = v_p$.
In this section we write $\mathrm{Sh}_{T,p}(v) := \mathrm{Sh}(T;v_0,\ldots,v_p)$ as a shorthand for the shuffle class associated with this path.
We also denote by $V(T)$ the set of vertices of a tree $T$.

For each $T$ in $\mathcal{C}_k^{(d)}$, consider now a probability measure $P_T$ supported on the set of vertices of height at least $p$ in $T$, i.e.\ a $[0,1]$-valued function on the vertices $v$ of $T$ such that 
\begin{align*}
\sum_{v \in V(T)} P_T(v) = 1,
\end{align*}
and $P_T(v) > 0$ if and only if $v$ has height at least $p$.
By the correspondence between vertices of $T$ of height at least $p$ and length-$p$ shuffle classes containing $T$, $P_T$ may also be regarded as a probability measure on the length-$p$ shuffle classes containing $T$.

\begin{df}
Fix a collection $\{P_T\colon T \in \mathcal{C}_k^{(d)}\}$, where each $P_T$ is a probability measure supported on the set of vertices of $T$ of height at least $p$.
The \textbf{$p$-shuffle chain} on $\mathcal{C}_k^{(d)}$ associated with $\{P_T\colon T \in \mathcal{C}_k^{(d)}\}$ is the Markov chain $(T_j)_{j \geq 0}$ with state space $\mathcal{C}_k^{(d)}$ and transition probabilities
\begin{align} \label{eq:tden}
\mathbb{P}\left( T_{j+1} = T' \,\middle|\, T_j = T \right) = \sum_{v \in V(T)} P_T(v) \frac{\mathbf{1}\{ T' \in \mathrm{Sh}_{T,p}(v) \} }{ \# \mathrm{Sh}_{T,p}(v) }.
\end{align}
In other words, conditional on the event $\{T_j = T\}$, in order to obtain $T_{j+1}$ we first choose a random vertex $v$ of $T$ according to $P_T(v)$, and then uniformly choose a tree $T_{j+1} = T'$ from those in $\mathrm{Sh}_{T,p}(v)$. 
\end{df}

Let us remark that choosing a tree $T'$ uniformly from those in $\mathrm{Sh}_{T,p}(v)$ is equivalent to taking a uniform permutation $\sigma$ of the $(d-1)p$ subtrees in the path of length $p$ leading up to $v$.

It is clear that, in the $p$-shuffle chain, communicating classes are closed: if there is a positive probability of going from $T$ to $T'$ in a certain number of steps, then there is also a positive probability of going from $T'$ to $T$ by just reversing these steps.
Furthermore, for fixed $d$ and $p$ and sufficiently large $k$, every state in a $p$-shuffle chain communicates with any other state.
Therefore, for large enough $k$, any $p$-shuffle chain is an irreducible Markov chain on a finite state space and, as such, it has a unique stationary distribution.

The framework of Section \ref{subsec:orthogonality} allows us now to show that the shuffle chain conjecture posed in the introduction implies the Jacobian conjecture.

\begin{proof}[Proof of Theorem \ref{thm:shufflechain}]
We first prove that Conjecture \ref{conj:shufflechain} implies Conjecture \ref{conj:shuffle}, and then appeal to Theorem \ref{thm:relation}. 

Let $d \geq 3$.
Suppose that, for all $p \geq 1$, there exists some $k_0(d,p)$ such that, for all $k \geq k_0(d,p)$, there exist a collection of probability measures $\{P_T\colon T \in \mathcal{C}_k^{(d)}\}$ with the property that the associated $p$-shuffle chain has the uniform distribution as its stationary distribution.
Then, for every tree $T'\in\mathcal{C}_k^{(d)}$, we have
\begin{align} \label{eq:mia}
\frac{1}{\# \mathcal{C}_k^{(d)}} = \sum_{T \in \mathcal{C}_k^{(d)} } \frac{1}{\# \mathcal{C}_k^{(d)}}\mathbb{P}\left( T_{j+1} = T' \,\middle|\, T_j = T \right) .
\end{align}
Multiplying \eqref{eq:mia} by $\# \mathcal{C}_k^{(d)}$ and using the transition probabilities \eqref{eq:tden}, we obtain 
\begin{align} \label{eq:mia2}
1 = \sum_{T \in \mathcal{C}_k^{(d)} } \sum_{v \in V(T)} P_T(v) \frac{\mathbf{1}\{ T' \in \mathrm{Sh}_{T,p}(v) \} }{ \# \mathrm{Sh}_{T,p}(v) }.
\end{align}
Given a shuffle class $\mathcal{S} \subseteq \mathcal{C}_k^{(d)}$, define the constant
\begin{align*}
\lambda_\mathcal{S} :=  \sum_{T \in \mathcal{C}_k^{(d)} } \sum_{v \in V(T)}  \frac{P_T(v)}{ \# \mathrm{Sh}_{T,p}(v) } \mathbf{1} \{\mathrm{Sh}_{T,p}(v)  = \mathcal{S} \} .
\end{align*}
Then, \eqref{eq:mia2} reads
\begin{align*}
1 = \sum_{ \mathcal{S} } \lambda_{\mathcal{S}} \mathbf{1}_{\mathcal{S}}(T')
\end{align*}
for all $T'$ in $\mathcal{C}_k^{(d)}$.
In other words, the constant function $\mathbf{1}:\mathcal{C}_k^{(d)} \to \mathbb{C}$ lies in the span of the length-$p$ shuffle indicators.

In particular, if, for some $d \geq 3$ and all $p \geq 1$, Conjecture \ref{conj:shufflechain} is true, then so is Conjecture \ref{conj:shuffle}, and, by virtue of Theorem \ref{thm:relation}, so is the Jacobian conjecture.
\end{proof}

Conjecture \ref{conj:shufflechain} is, strictly speaking, stronger than Conjecture \ref{conj:shuffle}, and corresponds to a special case of Conjecture \ref{conj:shuffle} in which all of the weights $\lambda_S$ are non-negative.
However, we believe it is a valuable reformulation, as it recasts the problem in terms of Markov chains on trees, for which there is a rich literature (see e.g.\ \cite{aldous2, AP, wilson, zachary}).

\section{Perfect trees} \label{sec:exponential}

In Section \ref{sec:intro} we briefly introduced the notion of $p$-perfect tree.
\begin{df}
A $d$-Catalan tree $T$ is \textbf{$p$-perfect} if it contains a path of vertices $(v_0,\ldots,v_p)$, as in Definition \ref{df:shuffle}, such that each of the $(d-1)p$ siblings of the vertices $v_1,\ldots,v_p$ is a leaf.
\end{df} 

See Figure \ref{fig:preperfect} for an illustration.

The main aim of this section is to prove that a large uniform $d$-Catalan tree is extremely likely to be $p$-perfect (Theorem \ref{thm:perfect}).
We then exploit this crucial fact to prove:
\begin{enumerate}
\item
the $\ell^1$ approximation of the constant function in Theorem \ref{thm:approx}, which can be regarded as an approximate version of Conjecture \ref{conj:shuffle};
\item certain quantitative estimates for the coefficients of the inverse of a Keller mapping, which can be regarded as an approximate version of the Jacobian conjecture itself.
\end{enumerate}

In recent decades, there has been a large body of literature on the local structure of random trees.
The interested reader might consult work by Aldous \cite{aldous}, Janson \cite{janson}, Holmgren and Janson \cite{HJ}, and in particular Janson's survey article \cite{jansonsurvey}.
Nonetheless, it appears the exact question we study (namely, estimating the probability that a large uniform $d$-Catalan tree is $p$-perfect) has not been examined in the literature.
Furthermore, for our statements we require absolute bounds on the probabilities involved, rather than convergence in distribution, hence we develop the requisite estimates from scratch.

\subsection{Uniform tree algorithm}\label{subsec:uniform}

We start by collecting some basic facts about $d$-Catalan numbers.
Denote by $C_k^{(d)}:= \#\mathcal{C}_k^{(d)}$ the \textbf{$d$-Catalan number}, i.e.\ the number of $d$-Catalan trees with $k$ internal vertices (and, thus, $dk+1$ vertices in total).
It is a combinatorial exercise (see for example \cite[Section 7.5]{concrete}) to prove that these numbers possess the closed-form expression
\begin{align} \label{eq:catcount}
\# \mathcal{C}_k^{(d)} = \frac{1}{(d-1)k+1} \binom{dk}{k}.
\end{align}
Using Stirling's formula, it is easily verified that
\begin{align} \label{eq:stirling}
C_k^{(d)} = (1 + o(1)) C_d k^{-3/2} e^{A_d k}
\qquad {\text{as } k\to\infty,}
\end{align}
where $C_d := \frac{1}{\sqrt{2  \pi }} \frac{ \sqrt{d}}{(d-1)^{3/2}}$ and $A_d := d \log d - (d-1) \log (d-1)$.

The $d$-Catalan numbers also satisfy the recursion
\begin{align*}
C_k^{(d)} = \sum_{ \substack{{k_1,\dots,k_d\geq 0}\colon \\k_1 + \cdots + k_d = k - 1 }} C_{k_1}^{(d)} \cdots C_{k_d}^{(d)} .
\end{align*}
This immediately sets up a way of generating a uniform Catalan tree in $\mathcal{C}_k^{(d)}$:
\begin{itemize}
\item {If $k=0$, make the root a leaf, thus obtaining the unique $d$-Catalan tree with $0$ internal vertices.
If $k\geq 1$, sample} a random vector $(k_1,\ldots,k_d)$, satisfying {$k_1,\dots,k_d\geq 0$ and} $k_1+\cdots + k_d = k -1$, from the probability distribution
\begin{align} \label{eq:leafer}
P_k^{(d)}(k_1,\ldots,k_d) :=  \frac{ C_{k_1}^{(d)} \cdots C_{k_d}^{(d)} }{C_k^{(d)}} .
\end{align}
\item {For all $i=1,\dots,d$}, tag the $i^{\text{th}}$ child (going left-to-right) of the root with the integer $k_i$.
If the $i^{\text{th}}$ child of the root is tagged with $k_i = 0$, make that child a leaf.
If the $i^{\text{th}}$ child is tagged with $k_i > 0$, then sample a new random vector $(k_{i,1},\ldots,k_{i,d})$, satisfying {$k_{i,1},\dots,k_{i,d}\geq 0$ and} $k_{i,1} + \cdots + k_{i,d} = k_i - 1$, from the probability distribution 
\begin{align*} 
P_{k_i}^{(d)}(k_{i,1},\ldots,k_{i,d}) :=  \frac{ C_{k_{i,1}}^{(d)} \cdots C_{{k_{i,d}}}^{(d)} }{C_{{k_i}}^{(d)}} .
\end{align*}
\item {Repeat the procedure with all the vertices in the second generation, and so on.}
\end{itemize}

\subsection{Likelihood of $p$-perfect trees}
\label{subsec:likelihoodPerfectTrees}

Let us now sketch how one can utilise the algorithm of Section \ref{subsec:uniform} to provide a description of the local structure of a $d$-Catalan tree at the root.
As in Section \ref{sec:intro}, let $\mathbb{P}_k^{(d)}$ be the uniform probability distribution on $\mathcal{C}_k^{(d)}$.
Sample $T$ from $\mathbb{P}_k^{(d)}$ and let $K_1,\ldots,K_d$ denote the respective numbers of internal vertices in the $d$ subtrees subtended by the $d$ children of the root.
The random vector $K := (K_1,\ldots,K_d)$ is then distributed according to \eqref{eq:leafer}.
Now write $(X_1,\ldots,X_{d-1})$ for the vector $K$ with the largest coordinate removed (if there are joint largest coordinates, remove any largest coordinate).
Whenever $x_1  + \cdots + x_{d-1} < k/2$ (in which case there certainly were not joint largest coordinates), we have
\begin{align*}
{\mathbb{P}_k^{(d)}}((X_1,\ldots,X_{d-1}) = (x_1,\ldots,x_{d-1} ) ) = d C_{x_1}^{(d)} \cdots C_{{x_{d-1}}}^{(d)} \frac{ C^{(d)}_{k - (x_1 + \cdots + x_{d-1} + 1) }}{ C_k^{(d)} } ,
\end{align*}
with the factor $d$ in front accounting for the fact that any given coordinate of $K$ is equally likely to be the largest (on the event that there is a unique largest coordinate). 
For fixed $(x_1,\ldots,{x_{d-1}})$, using \eqref{eq:stirling} we see that
\begin{equation} \label{eq:Qdist}
\begin{split}
Q_\infty ( x_1,\ldots,x_{d-1} )
&:= \lim_{k \to \infty} {\mathbb{P}_k^{(d)}}((X_1,\ldots,X_{d-1}) = (x_1,\ldots,x_{d-1} ) ) \\
&= d C_{x_1}^{(d)} \cdots C_{{x_{d-1}}}^{(d)} e^{ - A_d(x_1+\cdots+{x_{d-1}}+1)}.
\end{split}
\end{equation}
In other words, if we pick a uniform Catalan tree from $\mathcal{C}_k^{(d)}$, when $k$ is very large, we expect that, of the $d$ children subtended by the root, the subtrees subtended by exactly $d-1$ of these children usually have size $O(1)$, and one of the $d$ children subtends a subtree of size $k - O(1)$.
The distribution of the sizes of the $d-1$ order $1$ subtrees is given by \eqref{eq:Qdist}. 

Let us remark that such a local convergence holds in much more generality: Stufler \cite{stufler} showed that any conditional Galton--Watson tree seen from a typical vertex converges in any ball of radius $o(\sqrt{k})$ to a Kesten tree.

We hope this sketch calculation supplies the reader with an idea of the local behaviour of large $d$-Catalan trees.
Let us note in particular that the probability that, in a giant uniform $d$-Catalan tree (i.e.\ in the limit as $k \to \infty$),  all but one of the children of the root is a leaf, is given by 
\begin{align*}
Q_\infty ( 0,\ldots,0 ) = d e^{ - A_d} = { \left(1+\frac{1}{d-1}\right)^{-(d-1)}} \geq 1/e
\end{align*}
for all $d \geq 2$.
In turns out that {the} latter bound is in fact absolute across all $d$ and across uniform Catalan trees with any number of vertices:

\begin{lemma}\label{lem:rootChildren}
For all $d\geq 2$ and $k \geq 1$, the probability that at least $d-1$ of the children of the root of a uniformly chosen $d$-Catalan tree with $k$ internal vertices are leaves is at least $1/e$. 
\end{lemma}

\begin{proof}
If $k =1$, then all of the children are leaves with probability $1$, and we are done.

We may then suppose $k \geq  2$.
We have
\begin{align*}
\mathbb{P}^{(d)}_k (  \text{at least $d-1$ of the rootchildren are leaves} ) = d P_k^{(d)}(0,0,\ldots,0,k-1),
\end{align*}
where the factor of $d$ comes from the fact that there are $d$ different ways of choosing which of the $d$ children of the root is going to be the non-leaf. 

Note that $C_0^{(d)} = 1$.
Using \eqref{eq:leafer} {and \eqref{eq:catcount}}, we then have 
\begin{align*}
\mathbb{P}^{(d)}_k ( \text{at least $d-1$ of the rootchildren are leaves}) = d \frac{C_{k-1}^{(d)}}{C_k^{(d)}} = d \frac{(d-1)k+1}{(d-1)(k-1)+1} \frac{ \binom{d(k-1)}{(k-1)} }{ \binom{dk}{k} },
\end{align*}
which implies
\begin{align} \label{eq:wolf}
\mathbb{P}^{(d)}_k ( \text{at least $d-1$ of the rootchildren are leaves})  \geq  d \frac{ \binom{d(k-1)}{k-1} }{ \binom{dk}{k} }.
\end{align}
Expanding the binomial coefficients and rearranging, we see that
\[
d \frac{ \binom{d(k-1)}{k-1} }{ \binom{dk}{k} }
= \frac{  (dk-k)!}{ ((dk-k) - (d-1))!} \frac{ ((dk-1)-(d-1))!}{ (dk-1)!}
= \prod_{i=1}^{d-1} \frac{ dk - k - i +1 }{ dk - i } .
\]
{Noticing that the function $x\mapsto \frac{x-a}{x}$ is increasing on $(0,\infty)$ for any fixed $a>0$, we deduce}
\begin{align*} 
\mathbb{P}^{(d)}_k (  \text{At least $d-1$ of the rootchildren are leaves})
\geq \left(  \frac{ dk - k - d +2 }{ dk - d + 1} \right)^{d-1}.
\end{align*}
Now note that $ \frac{ dk - k - d +2 }{ dk - d + 1} = 1 - \frac{ (k-1)}{d(k-1)+1} \geq 1- 1/d$.
It follows that
\begin{align*}
\mathbb{P}^{(d)}_k (  \text{At least $d-1$ of the rootchildren are leaves}) \geq (1 - 1/d)^{d-1}.
\end{align*}
For $d \geq 2$, we have $(1-1/d)^{d-1}=(1+1/(d-1))^{-(d-1)} \geq 1/e$, completing the proof.
\end{proof}

\begin{cor}\label{cor:probPerfectChain}
Let {$k > p$}.
Then the $\mathbb{P}_k^{(d)}$-probability that there is at most one internal vertex in each of the generations {$1$ through $p$} is at least $e^{-p}$.
\end{cor}
\begin{proof}
Use the uniform tree algorithm and Lemma \ref{lem:rootChildren} $p$ times in a row.
\end{proof}

We now estimate the probability that $T$ is $p$-perfect.
Let $T$ be any tree in $\mathcal{C}_k^{(d)}$.
{For $d\geq 2$, }$T$ has at most $1 + d  + \cdots + d^{p-1 } \leq d^p$ {vertices in its generations $0$ through $p-1$}. 
Let $v_1,\ldots,v_r$ be a list of the vertices in generation $p$ of the tree, and let $n_1,\ldots,n_r$ denote the number of internal vertices in the trees they subtend.
Then
\begin{align}\label{eq:ineq}
n_1 + \cdots + n_r \geq k - d^p.
\end{align}
If we let $Q_{p,d}(k)$ be the probability that a uniformly chosen $d$-Catalan tree with $k$ {internal vertices} is not $p$-perfect. 
To obtain an estimate on $Q_{p,d}(k)$, note that we have
\begin{multline*}
Q_{p,d}(k)
\leq \; \mathbb{P}^{(d)}_k( \{ \text{$T$ does not have a {$p$-}perfect path starting at the root}\}  \\
   \cap \{\text{ $T$ does not have a {$p$-}perfect path starting at generation $p$}\} ) .
\end{multline*}
By Corollary \ref{cor:probPerfectChain}, the probability of the first event is bounded above by $1 - e^{-p}\leq e^{ - e^{-p} }$.
By the uniform tree algorithm, the probability of the second event, conditionally on $r$ and $n_1,\ldots,n_r$, is bounded above by $\prod_{i=1}^r Q_{p,d}(n_i)$.
Therefore, for any $k > p$ we have
\begin{align} \label{eq:relation}
Q_{p,d}(k)
\leq e^{ - e^{-p} } \sup_{ \substack{{r\leq d^p}, \\ n_1 + \cdots + n_r \geq k-d^p}} \prod_{i=1}^r Q_{p,d}(n_i),
\end{align}
where the supremum is taken over all $r\leq d^p$ (the number of vertices $r$ in generation $p$ of the tree satisfies $r \leq d^p$) and $r$-tuplets of non-negative integers $n_1,\ldots,n_r$ adding up to at least $k - d^p$ (by \eqref{eq:ineq}).
Using this inequality, we are now ready to prove Theorem \ref{thm:perfect}.

\begin{proof}[Proof of Theorem \ref{thm:perfect}]
We need to show that $Q_{p,d}(k)\leq e^{ - \kappa_{p,d} (k-p)_+}$ for all $k\in\mathbb{N}$. 
The result is clearly true for all $k \leq p$, because $Q_{p,d}(k)$ is a probability, and hence certainly $\leq 1$.
We now prove the result by induction for $k > p$.

Suppose that $Q_{p,d}(m)\leq e^{ - \kappa_{p,d} (m-p)_+}$ for all $m < k $.
Then, using \eqref{eq:relation}, we have 
\begin{align} \label{eq:relation2}
Q_{p,d}(k) \leq \sup_{ \substack{r\leq d^p, \\ n_1 + \cdots + n_r \geq k-d^p}}  e^{ - e^{-p} - \kappa_{p,d} \sum_{ i = 1}^r (n_i - p)_+ }.
\end{align}
Using the simple fact that $(n_i - p)_+ \geq n_i - p$ and the constraints, we obtain
\begin{align} \label{eq:relation3}
Q_{p,d}(k) \leq \exp \left\{ - e^{-p} - \kappa_{p,d} ( k - d^p - d^pp) \right\}.
\end{align}
To prove the result, it is now sufficient to show that 
\begin{align} \label{eq:relation5}
\exp \left\{ - e^{-p} - \kappa_{p,d} ( k - (p+1)d^p) \right\} \leq e^{ - \kappa_{p,d} (k-p)},
\end{align}
which amounts to the inequality
\begin{align*}
e^{-p} + \kappa_{p,d}(k-(p+1)d^p) \geq \kappa_{p,d}(k-p) .
\end{align*}
The latter is easily verified using the definition of $\kappa_{p,d}:= (2 p d^pe^{p})^{-1}$.
\end{proof}
Let us remark that, if $(v_0,\ldots,v_p)$ is a path of vertices in a tree, then $v_0,\ldots,v_{p-1}$ are certainly internal.
It follows that, when $p > k$, it is impossible to find a path of length $p$ in a tree with $k$ internal vertices.
Therefore, when $p > k$, $\mathcal{C}_k^{(d)}$ contains no $p$-perfect trees (or length-$p$ shuffle classes, for that matter) and, in this sense, the inequality of Theorem \ref{thm:perfect} is sharp.

\subsection{The $\ell^1$ bound}
\label{subsec:L1bound}

The existence of a function in the span of the length-$p$ shuffle indicators satisfying the total variation bound \eqref{eq:xx} in Theorem \ref{thm:approx} is an immediate consequence of Theorem~\ref{thm:perfect}, as we now see.

\begin{proof}[Proof of \eqref{eq:xx}]
If $T$ is $p$-perfect, then $T$ lies in a singleton length-$p$ shuffle class.
It follows that the indicator function of any $p$-perfect tree lies in the span of the length-$p$ shuffle indicators.
Therefore, the function
\begin{align*}
\phi (T) := \mathbf{1}\{ \text{$T$ is $p$-perfect} \}
\end{align*}
also lies in the span of the length-$p$ shuffle indicators.
Recall that $\mathbb{E}_k^{(d)}$ denotes the expectation with respect to $\mathbb{P}_k^{(d)}$.
It then follows from Theorem \ref{thm:perfect} that
\begin{align*}
|| \phi - \mathbf{1} ||_1
:= \mathbb{E}_k^{(d)} | 1 - \phi(T) | 
= \mathbb{P}_k^{(d)}( \text{$T$ is not $p$-perfect} ) \leq e^{ - \kappa_{p,d} (k-p) } ,
\end{align*}
as desired.
\end{proof}

The full proof of Theorem \ref{thm:perfect}, i.e.\ the exhibition of a function $\phi_k$ satisfying both \eqref{eq:xx} and \eqref{eq:yy}, can be found in Section \ref{sec:infinity}.

\subsection{Quantitative estimates for inverse coefficients of Keller mappings} \label{subsec:quantitative}

{The $\ell^1$ bound in Theorem \ref{thm:approx} yields {a bound} for the inverse coefficients $g_{i,\alpha}$ from \eqref{eq:Gsum} of maps $F$ of the form \eqref{eq:BCWform}.
{More precisely}, we will prove that the coefficients $g_{i,\alpha}$ satisfy a sharp bound, which, in the case where $F$ is a Keller mapping, can be improved upon by an exponential factor.}

\begin{thm} \label{thm:twobounds}
Let $F$ be a polynomial map of the form \eqref{eq:BCWform}, {where $JH$ may or may not be nilpotent.
Let $L:=\max_{i,\alpha} |H_{i,\alpha}|$.}
Then, for all $i \in [n]$ and all multi-indices $\alpha$ of degree $|\alpha| = (d-1)k+1$, we have
\begin{align} \label{eq:aa}
|g_{i,\alpha}| &\leq {\frac{1}{(d!)^k}} \frac{1}{(d-1)k+1} \binom{dk}{k}  n^{k-1} \frac{|\alpha|!}{ \alpha!}  L^k .
\end{align}
This bound is sharp, in the sense that there exists an $F$ such that the equality holds in \eqref{eq:aa} {for all $i$ and $\alpha$}.
If we additionally assume that $JH$ is nilpotent, with $(JH)^p = 0$ for some $p\in\mathbb N$, then for any $\phi$ lying in $\mathrm{span} \{ \mathbf{1}_\mathcal{S} \colon \mathcal{S}\text{ is a length-$p$ shuffle class in $\mathcal{C}_k^{(d)}$} \}$ we have
\begin{align} \label{eq:bb}
|g_{i,\alpha}|  &\leq {\frac{1}{(d!)^k}} \frac{1}{(d-1)k+1} \binom{dk}{k}  n^{k-1} \frac{|\alpha|!}{ \alpha!}  L^k \times || \phi - \mathbf{1} ||_1.
\end{align}
In particular, when $(JH)^p = 0$, {we have}
\begin{align} \label{eq:cc}
|g_{i,\alpha}|  &\leq {\frac{1}{(d!)^k}} \frac{1}{(d-1)k+1} \binom{dk}{k}  n^{k-1} \frac{|\alpha|!}{ \alpha!} L^k  \times C e^{ - ck},
\end{align}
{where $C$ and $c$ are the same constants appearing in \eqref{eq:xx}.}
\end{thm}

\begin{proof}
Let $H$ be a degree $d$ homogeneous polynomial mapping, where $JH$ may or may not be nilpotent.
{Notice that $L:=\max_{i,\alpha} |H_{i,\alpha}|$ is finite, since $H_{i,\alpha}=0$ for all but finitely many $i$ and $\alpha$.}
We thus have the simple bound 
\begin{align} \label{eq:energybound}
|\mathcal{E}_H(\mathcal{T})| \leq L^k
\end{align}
on the $H$-weight of a labelling $\mathcal{T}$ of a $d$-Catalan tree with $k$ internal vertices (see definition \eqref{eq:blender}).
In general, this bound is sharp, and is attained for any labelled tree $\mathcal{T}$ with $k$ internal vertices by setting $H_{i,\alpha} := L$ for all $i\in [n]$ and $|\alpha|=d$. 

Given $\alpha \in \mathbb{Z}^n_{\geq 0}$ and a $d$-Catalan tree $T$ with $|\alpha| = (d-1)k+1$ leaves, there are 
\begin{align} \label{eq:labellings}
\# \{ (i,\alpha)  \text{ labellings of $T$} \} = n^{k-1} \frac{|\alpha|!}{ \alpha!}
\end{align}
ways of $(i,\alpha)$-labelling $T$.
The factor $n^{k-1}$ is due to the number of ways of labelling the $k-1$ non-root internal vertices with any type in $[n]$, while $|\alpha|!/(\alpha_1! \cdots \alpha_n!)$ is the {multinomial coefficient counting the} distinct ways of labelling the leaves with labels in $[n]$ such that exactly $\alpha_j$ leaves are given label $j$.
We remark that the number of $(i,\alpha)$ labellings does not depend on the particular $d$-Catalan tree $T$, but only on its total number of vertices/leaves. 

Using \eqref{eq:labellings} and \eqref{eq:energybound} in \eqref{eq:F}, we obtain the bound
\begin{align} \label{eq:Esum}
|E_{i,\alpha,H}(T)| \leq  n^{k-1} \frac{|\alpha|!}{ \alpha!}  L^k
\end{align}
for all $d$-Catalan tree $T$ with $|\alpha| = (d-1)k+1$ leaves.
Using the bound \eqref{eq:Esum} in \eqref{eq:Gsum} in conjunction with $d$-Catalan number formula \eqref{eq:catcount}, we obtain
\begin{align} \label{eq:xx2}
|g_{i,\alpha}| \leq \frac{1}{(d!)^k}  \frac{1}{(d-1)k+1} \binom{dk}{k} n^{k-1} \frac{|\alpha|!}{ \alpha!}  L^k,
\end{align} 
which accounts for the first {bound} \eqref{eq:aa}.
The sharpness of \eqref{eq:xx2} (and \eqref{eq:Esum}) may be seen by setting again $H_{i,\alpha} := L$ for all $i,\alpha$; {one can verify that, for such a choice of $H$, $JH$ is \emph{not} nilpotent.}

Let us now suppose that {$JH$ is nilpotent, with} $(JH)^p = 0$, and turn to the proof of \eqref{eq:bb}.
Similarly as in the proof of Theorem \ref{thm:relation}, one can see that, whenever {$(JH)^p = 0$,} in light of \eqref{eq:kalpa0}, any function $\phi$ lying in the span of the length-$p$ shuffle indicators satisfies $\langle \phi, E_{i,\alpha,H} \rangle = 0$.
Accordingly, we have
\begin{align*}
\Phi_i^\alpha(H) = {\frac{1}{(d!)^k}} \langle \mathbf{1}, E_{i,\alpha,H} \rangle = {\frac{1}{(d!)^k}} \langle \mathbf{1} - \phi, E_{i,\alpha,H} \rangle.
\end{align*}
In particular, using \eqref{eq:Esum}, for any $\phi$ in the span of the length-$p$ shuffle indicators we have the total variation bound 
\begin{align} \label{eq:newbound}
\begin{split}
| \Phi_i^\alpha(H) |
&\leq {\frac{1}{(d!)^k}} n^{k-1} \frac{|\alpha|!}{ \alpha!} L^k  \sum_{T \in \mathcal{C}_k^{(d)} } | 1 - \phi(T) | \\
&= \frac{1}{(d!)^k}  \frac{1}{(d-1)k+1} \binom{dk}{k} n^{k-1} \frac{|\alpha|!}{ \alpha!}  L^k ||\phi - \mathbf{1}||_1 ,
\end{split}
\end{align}
which completes the proof of \eqref{eq:bb}.
{Finally, \eqref{eq:cc} follows immediately from \eqref{eq:bb} and \eqref{eq:xx}.}
\end{proof}

Let us take a moment to elaborate on this result.
First, we stress that the bound \eqref{eq:aa} concerns the inverse coefficients of \emph{any} degree $d$ homogeneous polynomial mapping and is not a priori concerned with the Jacobian conjecture.
We believe this bound, which is also new to the best of our knowledge, might be of independent interest.

As discussed in Section \ref{sec:Keller}, Bass, Connell and Wright showed in \cite{BCW} that, in order to prove the Jacobian conjecture, it is sufficient to consider maps $F=I-H$, with $H$ homogeneous of degree $d$ and $JH$ nilpotent (see Theorem \ref{thm:BassConnellWright}).
Indeed, the nilpotency condition is equivalent to the Keller condition for such maps (see \cite[Section 2.5]{JP}).

We remark that both bounds in \eqref{eq:aa} and \eqref{eq:cc} grow exponentially in $k$.
Theorem \ref{thm:twobounds} may be regarded as an approximate version of the Jacobian conjecture: the inverses of Keller maps are closer to being polynomial than their non-Keller counterparts, in that the coefficients of their high-degree terms {are smaller by an exponential factor.}

Finally, we note that \eqref{eq:bb} implies Theorem \ref{thm:relation}.
Indeed, if Conjecture \ref{conj:shuffle} is true, then, for large enough $k$, we can set $\phi = \mathbf{1}$ in \eqref{eq:bb}, so that the bound yields $g_{i,\alpha} = 0$ whenever $|\alpha| = (d-1)k+1$.

\section{Construction of the $\ell^\infty$ bound} \label{sec:infinity}

{In this section, we construct a function that satisfies the $\ell^{\infty}$ bound of Theorem \ref{thm:approx} and conclude its proof.}

\subsection{Indexing of shuffle classes by vertices and a first approximation} \label{subsec:first}

Through this section, we will use the notation
\begin{align*}
N_j(T) := \# \{ \text{vertices in generation $j$ of $T$} \}
\end{align*}
for the number of vertices at graph distance $j$ from the root of {a tree} $T$.
We also write
\begin{align*}
N_{ \leq q}(T) := \sum_{0 \leq j \leq q} N_j(T) \qquad \text{and} \qquad N_{ \geq q}(T) := \sum_{j\geq q} N_j(T), 
\end{align*}
where the latter sum is clearly finite for any finite tree $T$. 

Recall that, given a tree $T$ and a path of vertices $(v_0,\ldots,v_p)$ in $T$, we let the shuffle class $\mathrm{Sh}(T;v_0,\ldots,v_p)$  be the set of trees $T'$ that may be obtained from $T$ by shuffling the subtrees {subtended by the siblings of $v_1,\ldots,v_p$.}
Consider now examining the set of shuffle classes containing a given tree $T$.
{Each} vertex $v$ lying at a height of at least $p$ in $T$ may be associated with a canonical path of $(v_0,\ldots,v_p)$ of vertices in $T$ by letting $v_p := v$ and $v_{p-(i+1)}$ be the unique parent of $v_{p-i}$, for each {$0 \leq i \leq p-1$}.
Therefore, vertices of height at least $p$ are in bijection with the set of paths of length $p$.
In particular, for each tree $T$, we might informally expect in a suitable sense {that}
\begin{align} \label{eq:counting}
\# \{ \text{{length-$p$} shuffle classes containing $T$} \} \approx N_{\geq p}(T).
\end{align}
We need to be careful however: it is entirely possible that two different choices $v$ and $v'$ of vertices of height at least $p$, and their associated paths, give rise to the \emph{same} shuffle class.
As a simple example, consider a tree $T$ containing two different perfect $p$-paths $(v_0,\ldots,v_p)$ and $(v'_0,\ldots,v'_p)${: then, both paths (or equivalently, both vertices} $v=v_p$ and $v'=v_p'$) give rise to the same shuffle class, i.e.\ the singleton $\{T\}$. 

It is however possible to make precise sense of \eqref{eq:counting} after taking on a suitable indexing of the shuffle classes.
{To do so, we } begin by noting that each vertex $v$ in a $d$-Catalan tree may be associated with a code in $\cup_{ j \geq 0} \{1,\ldots,d\}^j$: the root is given by the empty code, the $d$ children of the root, listed from left to right, are given by the $1$-tuples $(1),\ldots,(d)$, and, more generally, the children of an internal vertex with code $u$ are given by the concatenations $(u1),\ldots,(ud)$.
The length of the code of a vertex is the same as its height in the tree.
A $d$-Catalan tree itself may, as such, be regarded as a subset of $\cup_{ j \geq 0} \{1,\ldots,d\}^j$ {such that, if $(ui)\in T$ for some code $u$, then $(u)\in T$ and $(uj)\in T$ for all $j=1,\dots,d$}.
Let $(v_0,\ldots,v_p)$ be a path in a $d$-Catalan tree $T$, which we may then associate with a collection of codes of lengths $j,\ldots,j+p$ for some $j \geq 0$.
We note that every tree $T'$ in the shuffle class $\mathrm{Sh}(T;v_0,\ldots,v_p)$ contains this {collection} of codes.

We now define a function $\tilde{\psi}$ lying in the span of the indicator functions of the length-$p$ shuffle classes by letting
\begin{align} \label{eq:psidef}
\tilde{\psi} := \sum_{ j \geq p} \sum_{ v \in \{1,\ldots,d\}^j } \sum_{ \substack{\mathcal{S} = \mathrm{Sh}(T;v_0,\ldots,v_p{=v}) \\ \text{ for some $T\in\mathcal{C}_k^{(d)}$} }} \mathbf{1}_\mathcal{S}.
\end{align}
We now note that
\begin{align} \label{eq:psieq}
\tilde{\psi}(T) = N_{\geq p}(T).
\end{align}
Indeed, for every tree $T$, every vertex $v$ in $T$ of height at least $p$ (which has a unique code of length at least $p$) scores a $1$ in the sum $\tilde{\psi}$.
It follows that \eqref{eq:psidef} and \eqref{eq:psieq} make sense of \eqref{eq:counting}.

{In the following, our goal is to show that,} when $k$ is large compared to $p$, the vast majority of vertices in any tree $T$ with $k$ internal vertices have height at least $p$.
In fact, we can be far more specific: every tree $T$ in $\mathcal{C}_k^{(d)}$ has {at most $1 +d+ \cdots + d^{p-1}$} vertices in its generations $0$ through $p-1$, hence
\begin{align} \label{eq:sandwich}
{N_{\geq p}(T) \geq (dk+1) - (1 + d+\cdots + d^{p-1})}
\end{align}
In particular, we are equipped to give the following preliminary construction in the direction of the uniform bound {of Theorem \ref{thm:approx}.}

\begin{proposition}
Define 
\begin{align*}
\psi := \frac{1}{dk+1} \sum_{ j \geq p} \sum_{ v \in \{1,\ldots,d\}^j } \sum_{ \substack{\mathcal{S} = \mathrm{Sh}(T;v_0,\ldots,v_p=v) \\ \text{ for some $T\in\mathcal{C}_k^{(d)}$} }} \mathbf{1}_\mathcal{S}.
\end{align*}
Then $\psi$ lies in the span of the length-$p$ shuffle indicators, and satisfies the uniform bound
\begin{align} \label{eq:presup}
\mathrm{sup}_{T \in \mathcal{C}_k^{(d)}} | \psi(T) - 1 | \leq C_{d,p} /(dk+1),
\end{align}
where we take $C_{d,p} :=(d^p-1)(d-1)^{-1}$. 
\end{proposition} 
\begin{proof}
Clearly $\psi$ is a scalar multiple of $\tilde{\psi} = N_{ \geq p}$, and hence lies in the span of the length-$p$ shuffle indicators.
The bound \eqref{eq:presup} then follows from \eqref{eq:sandwich}.
\end{proof}

The previous result exhibits a function $\phi$ in the {span of the} length-$p$ shuffle indicators, for which the supremum norm $|| \phi - \mathbf{1}||_\infty$ decays like $1/k$.
In the next section, we develop {a more sophisticated approach to find} a function $\phi$ in the span of the length-$p$ shuffle indicators for which $|| \phi - \mathbf{1}||_\infty$ is bounded superpolynomially in $1/k$.

\subsection{Construction of the width product function}
\label{subsec:widthProduct}

We refine the ideas of Section \ref{subsec:first} with the following lemma.

\begin{lemma}\label{lem:prodShuffle}
Let ${\chi}\colon\mathcal{C}_k^{(d)} \to \mathbb{R}$ be a function that only depends on generations $0,1,\ldots,j$ of a $d$-Catalan tree{; namely,} if $T$ and $T'$ have the same generations $0,\ldots,j$, then $\chi(T) = \chi(T')$.
Then, the function $T\mapsto N_{j+p}(T) \chi(T)$ lies in the span of the length-$p$ shuffle indicators.
\end{lemma}

\begin{proof}
Let $v$ be a vertex of height $j+p$ in a tree $T$, and consider {the path $(v_0,\dots,v_p)$, with $v_p=v$, and the associated shuffle class $\mathcal{S} := \mathrm{Sh}(T; v_0,\dots,v_p)$.}
Then, $v_0$ has height $j$.
We note that shuffling the subtrees off the path $(v_0,\ldots,v_p)$ leaves generations $0,\ldots,j$ of the tree unaffected and, accordingly, $\chi$ is constant on $\mathcal{S}$.
In particular, the function $\mathbf{1}_\mathcal{S}(T) \chi(T)$ lies in the span of the length-$p$ shuffle indicators.
Summing over all such shuffle classes obtained this way, we obtain the function $N_{j+p}(T)\chi(T)$. 
\end{proof}

\begin{cor}\label{cor:prodShuffle}
For each integer $m \geq 1$, the function
\begin{align*}
{J_m(T):=} N_{ \geq m(p-1)+1}(T) \prod_{i=1}^{m-1}N_{\leq i(p-1)}(T)
\end{align*}
lies in the span of the {length-$p$ shuffle} indicators.
\end{cor}

\begin{proof}
{Let $\chi(T) := \prod_{i=1}^{m-1}N_{\leq i(p-1)}(T)$.
Then, by Lemma \ref{lem:prodShuffle},} $N_{j+p}(T)\chi(T)$ lies in the span of the length-$p$ shuffle indicators whenever {$j \geq (m-1)(p-1)$}.
Equivalently, $N_k(T) \chi(T)$ lies in the span of the shuffle indicators for all $k \geq m(p-1)+1$.
Summing over all such $k$, it follows that 
\begin{align*}
{\sum_{k \geq m(p-1)+1} N_k(T) \chi(T)
=N_{ \geq m(p-1)+1}(T) \prod_{i=1}^{m-1}N_{\leq i(p-1)}(T)}
\end{align*}
lies in the span of the length-$p$ shuffle indicators.
\end{proof}

Note that, since each tree in $\mathcal{C}_k^{(d)}$ has $dk+1$ vertices, we have the crude bound
\begin{align*}
J_m(T)
\leq ( dk+1)^m \mathbf{1} \{ \text{the height of $T$ is at least $m(p-1)+1$} \}.
\end{align*}
Since a tree in $\mathcal{C}_k^{(d)}$ has height at most {$k$}, $J_m(T)$ is identically zero {for large enough $m$.}
In particular, we have a well-defined infinite sum
{\begin{align*}
\begin{split}
\phi(T) &:= \sum_{m=1}^\infty \frac{J_m(T)}{ (dk+1)^m} \\
&= \sum_{m=1}^\infty \frac{(dk+1-N_{\leq m(p-1)}) \prod_{i=1}^{m-1}N_{\leq i(p-1)}(T)}{(dk+1)^m} \\
&= \sum_{m=1}^\infty \left(\prod_{i=1}^{m-1}\frac{N_{\leq i(p-1)}(T)}{dk+1} - \prod_{i=1}^{m}\frac{N_{\leq i(p-1)}(T)}{dk+1}\right).
\end{split}
\end{align*}
Since the latter is a telescopic sum, we obtain the expression
\begin{align} \label{eq:prod}
\phi(T) = 1 - \prod_{m=1}^\infty \frac{ N_{\leq m(p-1)}(T) }{ dk+1}.
\end{align}}
Again, the product may be taken over finitely many terms: since each tree $T$ in $\mathcal{C}_k^{(d)}$ has height at most {$k$}, we have that $N_{\leq m(p-1)}(T) / ( dk+1) =1$ whenever $m(p-1) \geq k$. 
Notice also that, since all $J_m$ lie in the span of the length-$p$ shuffle indicators by Corollary \ref{cor:prodShuffle}, $\phi$ also lies in the same span.
We now provide an upper bound on the quantity $||\phi-\mathbf{1}||_\infty$. 

\begin{lemma} \label{lem:sup}
Let {$d\geq 2$} and $k\in\mathbb N$ such that {$k\geq d^{6p}$.}
Then, for every $T \in \mathcal{C}_k^{(d)}$, we have 
$|1 - \phi(T)| \leq \exp ( - c_{d,p}  \log^2k )$, where $c_{d,p}\in(0,\infty)$ may be taken as $(4 p \log d)^{-1}$. 
\end{lemma}

\begin{proof}
If $T$ in $\mathcal{C}_k^{(d)}$, then $N_j(T) \leq d^j$.
In particular, 
\[
N_{\leq m(p-1)}(T) = \sum_{ {j = 0}}^{m(p-1)} N_j(T) \leq \sum_{ j = 0}^{m(p-1)} d^j = \frac{ d^{m(p-1) + 1} - 1}{d-1} \leq d^{mp}
\qquad \text{for } d\geq 2. 
\]
Also, trivially we have $N_{\leq m(p-1)}(T) \leq dk+1$. 
It follows that
\begin{align*}
|1 - \phi(T)| = \prod_{m=1}^\infty \left|  \frac{ N_{\leq m(p-1)}(T) }{ dk+1} \right| \leq \prod_{m=1}^\infty \left| \frac{ d^{mp}}{ dk+1} \wedge 1 \right| \leq \prod_{m=1}^\infty \left| \frac{ d^{mp}}{ k} \wedge 1 \right| = \prod_{m=1}^{m_0}  \frac{ d^{mp}}{ k}  , 
\end{align*}
where in the final expression above, $m_0$ denotes the largest integer $m$ such that $d^{mp} \leq k$.
Then 
\begin{align*}
{m_0 = \left\lfloor \frac{ \log k}{ p \log d} \right\rfloor .}
\end{align*}
A calculation tells us that $ \prod_{m=1}^{m_0}  \frac{ d^{mp}}{ k}  = \exp \left( - m_0 \log k + p (\log d) m_0 (m_0+1)/2 \right)$.
Now, using the simple bound {$\frac{ \log k}{ p \log d}-1\leq m_0 \leq  \frac{ \log k}{ p \log d} $}, we obtain 
\begin{align*}
\begin{split}
- m_0 \log k + p \log d \frac{m_0 (m_0+1)}{2} 
&\leq  -  \frac{ \log^2 k}{ p \log d}  + \log k +  \frac{1}{2} \frac{(\log k)^2}{p \log d} + \frac{1}{2}  \log k \\
&=  -  \frac{1}{2} \frac{ \log^2 k}{ p \log d}  + \frac{3}{2} \log k.
\end{split}
\end{align*}
Now, since by assumption $\log (k) \geq {6p \log(d)}$, we arrive at
\begin{align*}
- m_0 \log k + p \log d \frac{m_0 (m_0+1)}{2}   &\leq  - \frac{ \log^2k}{ 4p \log d} ,
\end{align*}
thus completing the proof.
\end{proof}

While Lemma \ref{lem:sup} provides a uniform upper bound on $|1-\phi(T)|$ on all trees $T$ in $\mathcal{C}_k^{(d)}$, {we remark that,} for the typical tree $T$ in $\mathcal{C}_k^{(d)}$, the quantity {$|1-\phi(T)|$} is much smaller.
Sketching a few of the key steps here, it is known that, when $k$ is large, the typical tree $T$ in $\mathcal{C}_k^{(d)}$ has both height and width of order $\sqrt{k}$; see for example Addario-Berry, Devroye and Janson \cite{ADJ}.
Roughly speaking, this entails that
\begin{align} \label{eq:zz}
{|1 - \phi(T)|} =  \prod_{m=1}^\infty  \frac{ N_{\leq m(p-1)}(T) }{ dk+1} \sim \prod_{m=1}^{O(\sqrt{k}/p)} \frac{ O(\sqrt{k})}{ dk+1 } = e^{ - O\left( {\frac{\sqrt{k}}{p}} \log k \right)}.
\end{align}
It is possible to make \eqref{eq:zz} rigorous using the absolute bounds in \cite{ADJ}.
We leave the details to the interested reader.

\subsection{Proof of Theorem \ref{thm:approx}}
\label{subsec:proofApproximation}

In Section \ref{sec:exponential} we saw that a proportion of at most $e^{ -\kappa_{d,p} {(k-p)_+}}$ of the trees in $\mathcal{C}_k^{(d)}$ {are not $p$-perfect}.
So far in this section, we have seen that the function $\phi$ given by \eqref{eq:prod} has the property that $|1-\phi(T)| \leq e^{ - c_{d,p} (\log k)^2}$.
We now show it is possible to combine these two results to prove {Theorem \ref{thm:approx}.}

\begin{proof}[Proof of Theorem \ref{thm:approx}]

We begin by noting that, given \emph{any} function $\rho:\mathcal{C}_k^{(d)} \to \mathbb{R}$, the function 
\begin{align*}
{\rho(T)}\mathbf{1}\{T \text{ is $p$-perfect}\} 
\end{align*}
lies in the span of the length-$p$ shuffle indicators, since it is a linear combination of the indicator functions of {individual} $p$-perfect trees, which lie in singleton shuffle classes.

{Therefore, taking $\rho:=\phi$ as in \eqref{eq:prod}, both $(1 - \phi(T)) \mathbf{1}\{ T \text{ is $p$-perfect} \}$ and $\phi$ lie} in the span of the length-$p$ shuffle indicators.
In particular, so does the sum of these two functions, which is
\begin{align} \label{eq:phistar}
\phi_*(T) := \mathbf{1} \{ T \text{ is $p$-perfect} \} + \mathbf{1} \{ T \text{ is not $p$-perfect} \} \phi(T).
\end{align}
{We clearly have
\[
||\phi_*(T) - \mathbf{1} ||_\infty \leq ||\phi(T) - \mathbf{1} ||_\infty \leq e^{ - c_{d,p} (\log k)^2 }
\qquad\quad \text{for } k\geq d^{6p}
\]}
by Lemma \ref{lem:sup}, proving that $\phi_*$ satisfies \eqref{eq:yy}.

Moreover, for all $k\geq d^{6p}$, we have
\begin{align*}
||\phi_*(T) - \mathbf{1} ||_1 &:= \frac{1}{ \# \mathcal{C}_k^{(d)} } \sum_{T \in \mathcal{C}_k^{(d)} } | 1 - \phi_*(T) |\\
&= \frac{1}{ \# \mathcal{C}_k^{(d)} } \sum_{T \in \mathcal{C}_k^{(d)} \text{ not $p$-perfect} } | 1 - \phi(T) |\\
&\leq \frac{ \# \{ T \in \mathcal{C}_k^{(d)}\colon \text{$T$ is not $p$-perfect} \} }{ \# \mathcal{C}_k^{(d)} } e^{ - c_{d,p} (\log k)^2 } \\
& \leq e^{ - c_{d,p} (\log k)^2 - \kappa_{d,p} (k-p)}.
\end{align*}
In the second equality above we used the definition \eqref{eq:phistar}, in the following inequality we used Lemma~\ref{lem:sup}, whereas in the final inequality we used Theorem~\ref{thm:perfect}.
{Since $k-p$ diverges to $\infty$ faster than $(\log k)^2$ as $k\to\infty$, we have that $\phi_*$ also satisfies the $\ell^1$ bound \eqref{eq:xx}.}
\end{proof}


\section*{Acknowledgements}

This research was supported by the Heilbronn Institute for Mathematical Research (HIMR) as part of the HIMR Focused Research workshop titled: \emph{A graph-theoretic approach to the Jacobian conjecture}.

Piotr Dyszewski was partially supported by the National Science Centre, Poland (Sonata, grant number 2020/39/D/ST1/00258).

Samuel G.\ G.\ Johnston was partially supported in the early stages by the EPSRC funded Project EP/S036202/1 \emph{Random fragmentation-coalescence processes out of equilibrium}.

Joscha Prochno was supported by the Austrian Science Fund (FWF) Project P32405 \emph{Asymptotic geometric analysis and applications}.

Dominik Schmid acknowledges the DAAD PRIME program for financial support.

\end{document}